\documentclass[12pt]{amsart}
\usepackage{amsthm, amsmath, amssymb, graphicx, array, mathrsfs, textcomp, setspace, bbding, amsthm, latexsym, amscd}
\newcommand{\K}{{\mathbb K}}

\newcommand{\C}{{\mathbb{C}}}

\newcommand{\Rn}{{\R^n}}
\newcommand{\N}{{\mathbb{N}}}
\newcommand{\s}{{\mathrm{S}}}
\newcommand{\sn}{{\s^{n-1}}}

\newcommand{\Vol}{{\mathrm{Vol}}}

\newcommand{\sw}{{\mathcal S}}
\newcommand{\RT}{{\mathcal R}}

\newcommand{\SO}{\mathrm{SO}}

\renewcommand{\:}{\, : \,}

\begin{document}
\newcommand{\R}{{\mathbb R}}
\newcommand{\Z}{{\mathbb Z}}
\newcommand{\trace}{\rm trace}
\newcommand{\Ex}{{\mathbb{E}}}
\newcommand{\Prob}{{\mathbb{P}}}
\newcommand{\E}{{\cal E}}
\newcommand{\F}{{\cal F}}
\newtheorem{df}{Definition}
\newtheorem{theorem}{Theorem}
\newtheorem{lemma}{Lemma}
\newtheorem{pr}{Proposition}
\newtheorem{co}{Corollary}
\newtheorem{re}{Remark}
\newcommand{\sign}{\mbox{ sign }}
\newcommand{\A}{{\cal A}}
\newcommand{\X}{{\cal X}}
\newcommand{\diam}{\mbox{\rm dim}}
\newcommand{\vol}{\mbox{\rm Vol}}
\newcommand{\e}{\varepsilon}

\title{Sections of Convex Bodies with Symmetries}
\author{Susanna Dann}
\address{Institute of Discrete Mathematics and Geometry, Vienna University of Technology, Wiedner Hauptstrasse 8-10, 1040 Vienna, Austria.}
\email{susanna.dann@tuwien.ac.at}

\author{Marisa Zymonopoulou}
\address{Department of Mathematics, University of Athens, Panepistimioupolis 15784, Athens, Greece.}
\email{marisa.zym@gmail.com}

\begin{abstract}
In this paper we study how certain symmetries of convex bodies affect their geometric properties. In particular, we consider the impact of symmetries generated by the block diagonal subgroup of orthogonal transformations, generalizing complex and quaternionic convex bodies. We conduct a systematic study of sections of bodies with symmetries of this type, with the emphasis on problems of the Busemann-Petty type and hyperplane inequalities. The main role belongs to the class of intersection bodies with symmetries. 
\end{abstract}
\maketitle
\section*{Introduction}

Convex bodies with symmetries, especially bodies of revolution, have been serving as a proving ground in the study of sections for a long time. The convenience of such bodies is based on the explicit formulas for the volume of their sections. Formulas of this type were probably first exploited by H. Hadwiger, who provided a positive answer to the Busemann-Petty problem in $\R^3$ for origin-symmetric axially convex bodies of revolution in \cite{MR0254739}. In connection with the Busemann-Petty problem in $\R^n$, bodies of revolution were used by A. Giannopoulos in \cite{Gi} to supply a negative answer for $n\geq 7$, by M. Papadimitrakis \cite{Pa} to prove a negative answer for $n=5, 6$, by R. Gardner \cite{Ga1} to prove a negative answer for $n \geq 5$, and by G. Zhang \cite{Zh1} to provide further counterexamples. Other results on sections and projections of bodies with symmetries are due to E. Milman \cite{milman2008}, B. Rubin \cite{MR2570666}, B. Rubin and G. Zhang \cite{MR2078635}, D. Ryabogin and A. Zvavitch \cite{MR2102643} and many others.

A convex body in $\C^n$ is a convex body in $\R^{2n}$ that is invariant under the block diagonal subgroup of $\SO(2n)$ of the form $$ \{ \mathrm{diag}(g,\dots, g) \: g\in \SO(2) \} \,, $$
where $\SO(\cdot)$ stands for the special orthogonal group over the reals. Many properties of sections can be improved by imposing the complex structure, in other words by imposing the invariance under the above group, see \cite{MR3106735}.

In this paper we generalize the results from \cite{MR3106735} by studying the impact of the following group of symmetries: The block diagonal subgroup of $\SO(\kappa n)$ of the form
$$ \{ \mathrm{diag}(g,\dots, g) \: g\in \SO(\kappa) \} \,, $$
where $\kappa \in \N$ is fixed. Subsets of $\R^{\kappa n}$ that are invariant under the above group will be called \text{\it $\kappa$-balanced}. To stress the fact that we work with this fixed group of symmetries, the space $\R^{\kappa n}$ along with $\kappa$-balanced geometric objects in this space (such as star shaped bodies, linear subspaces, etc.) will be denoted by $\K^n$. For $\kappa = 1, 2, 4$, $\K^n$ can be thought of as the $n$-dimensional real, complex or quaternionic vector space, respectively; however our results hold in more generality for any $\kappa \in \N$.\footnote{For $\kappa = 1$, $\K^n=\R^n$ and for $\kappa = 2$, $\K^n =\C^n$. For $\kappa = 4$, after some additional restrictions on our group of symmetries, $\K^n$ will become the left or right quaternionic vector space. However, such restrictions are not natural and not necessary when considering geometric problems.}

To study properties of convex bodies with symmetries we introduce the concept of intersection bodies in $\K^n$. Recall that intersection bodies in $\R^n$ were introduced by E. Lutwak in 1988 as part of his dual Brunn-Minkowski theory \cite{Lu}. For a subset $S$ of $\R^n$, define the \text{\it Minkowski functional} of $S$ by
$$ \|x\|_S := \inf \{ \lambda > 0 \: x\in \lambda S\} \, ,$$
with $x\in \R^n$. An origin-symmetric star body $K$ in $\R^n$ is the \text{\it intersection body} \text{\it of an origin}-\text{\it symmetric star body $L$} if the radius of $K$ in every direction equals to the $(n-1)$-dimensional volume of the central hyperplane section of $L$ perpendicular to this direction. In other words, for every unit vector $\xi$ in $\R^n$, 
\begin{equation}\label{eq_ribosb}
\|\xi\|_K^{-1} = |L\cap \xi^{\perp}| \, ,
\end{equation}
where $|\cdot|$ denotes the Euclidean volume. Using polar coordinates, equation (\ref{eq_ribosb}) becomes
$$ \|\xi\|_K^{-1} = \frac{1}{n-1} \int_{\s^n\cap \xi^{\perp}} \|\theta\|_L^{-n+1} d\theta = \frac{1}{n-1} \RT_{n-1} (\| \cdot \|_L^{-n+1}) (\xi) \, , $$
where $\RT_{n-1}$ denotes the spherical Radon transform. Hence, a star body $K$ in $\R^n$ is the intersection body of a star body if and only if $\|\cdot\|_K^{-1}$ is the spherical Radon transform of a continuous positive function on $\s^{n-1}$. 

A more general class of intersection bodies in $\R^n$ was introduced by P. Goodey, E. Lutwak and W. Weil in 1996 in \cite{GoodeyLutwakWeil1996}. A star body $K$ is an \text{\it intersection body} if there exists a finite non-negative Borel measure $\mu$ on the sphere so that $\|\cdot\|_K^{-1}=\RT_{n-1} \mu$. 

The class of intersection bodies in $\R^n$ has been investigated in \cite{borell73, busemann49, Busemann53,	MR1695720, MR1750399,  Ga1,Ga2,GardnerKoldobskySchlumprecht1999, GoodeyLutwakWeil1996, GoodeyWeil1995, hensley80, KaltonKoldobsky05, MR1750442, MR2826412, MR1285999, MR1637955, K2, MR1623670, MR1694767, Koldobsky2000, ludwig2006, Lu, milman2006, milman2008, Zh1, Zh2, Zhang1996}, see also \cite{MR2251886, Koldobsky2005, KoldobskyYaskin2008}. An analogous class of bodies in $\C^n$ was studied in \cite{MR3106735}. Intersection bodies in $\C^n$ were defined along the same lines as intersection bodies in $\R^n$, taking into account the complex structure. They inherit many properties of their real counterparts.

In our discussion we follow ideas from \cite{MR3106735}. We generalize to $\K^n$ many known results from the theory of intersection bodies in $\R^n$ and $\C^n$. We organized this paper as follows. In Section \ref{SectionIntersectionBodiesOfStarBodies} we define intersection bodies of star bodies in $\K^n$. In Section \ref{SectionRadonFourierTransforms} we introduce the spherical Radon transform on $\K^n$ and prove that it coincides with the Fourier transform of distributions on the class of $(-\kappa n + \kappa)$-homogeneous functions on $\R^{\kappa n}$ that are $\kappa$-invariant, see Lemma \ref{lemma_rfts}. This allows to express the volume of sections of star bodies in $\K^n$ in Fourier analytic terms, see Theorem \ref{thm_vos}. Intersection bodies in $\K^n$ are introduced in Section \ref{SectionIntersectionBodies}; here we also prove their Fourier analytic characterization in Theorem \ref{th_kibpdd}. In Section \ref{section_cibinkn} we use the above characterization to show that intersection bodies in $\K^n$ coincide with a generalization of real intersection bodies due to A. Koldobsky: the $\kappa$-balanced $\kappa$-intersection bodies in $\R^{\kappa n}$, see Corollary \ref{co_ibkib}. In Corollary \ref{co_ib} we list all the cases in which an origin-symmetric convex body in $\K^n$ is an intersection body in $\K^n$, these are only the following: 
$$ \text{ (i) } n=2, \kappa\in \N, \text{ (ii) } n=3, \kappa\leq 2 \text{ and (iii) } n=4, \kappa =1.$$
Next, we extend to $\K^n$ a result of P. Goodey and W. Weil that intersection bodies in $\R^n$ can be obtained as the closure in the radial metric of radial sums of ellipsoids, see Theorem \ref{th_ibaloe}. We use this geometric characterization of  intersection bodies in $\K^n$ to prove that they coincide with another generalization of real intersection bodies due to G. Zhang: $\kappa$-balanced generalized $\kappa$-intersection bodies in $\R^{\kappa n}$, see Proposition \ref{pr_ibgkib}. We start Section \ref{section_StabilityInTheBusemannPettyProblemAndHyperplaneInequalities} by solving the Busemann-Petty problem in $\K^n$, see Theorem \ref{th_bp}, and by deriving the hyperplane inequality in $\K^n$: Suppose $K$ is an intersection body in $\K^n$, then
$$ |K|^{\frac{n-1}{n}} \leq  \frac{|B^{\kappa n}_2|^{\frac{n-1}{n}}}{|B^{\kappa n-\kappa}_2|} \max_{\xi\in\s^{\kappa n -1}} |K\cap H_{\xi}| \, ,$$
where $H_{\xi}$ stands for a hyperplane in $\K^n$, see Corollary \ref{co_hibp}. Theorem \ref{th_bppokn} deals with the Busemann-Petty problem in $\K^n$ for arbitrary measures. From the stability consideration in this problem we derive in Theorem \ref{th_hifam} the hyperplane inequality for intersection bodies in $\K^n$ for arbitrary measures. In Lemma \ref{le_hiamdfg} and its corollaries we describe further inequalities obtained from the stability consideration mentioned above; here we take advantage of the fact that we solve the stability question with different density functions for the volume of the body and the volume of sections. Finally, intersection bodies of convex bodies in $\K^n$ are studied in Section \ref{section_IntersectionBodiesOfConvexBodies}; here, in Theorem \ref{th_bt} and Corollary \ref{co_hbt} we extend to $\K^n$ two classical results about intersection bodies of convex bodies in $\R^n$: Busemann's and Hensley-Borell theorems. We introduce the notation and preliminaries throughout the article as needed.

\section{Intersection Bodies of Star Bodies in $\K^n$}\label{SectionIntersectionBodiesOfStarBodies}

Let $\kappa \in \N$ and $x = (x_1, x_2, \dots, x_{\kappa n}) \in \R^{\kappa n}$. We view $x$ as an ordered set of $n$ ordered $\kappa$-tuples. For every $\sigma \in \SO(\kappa)$ define 
$$ R_{\sigma} (x):=(\sigma(x_1, \dots, x_{\kappa}), \dots, \sigma(x_{\kappa (n-1)+1}, \dots, x_{\kappa n}))$$
to be the vector obtained by rotating the ordered $\kappa$-tuples of $x$. A set $M$ in $\R^{\kappa n}$ is called \text{\it $\kappa$-balanced} if 
$$ \|x\|_{M} = \|R_{\sigma} (x)\|_M = \|\sigma(x_1, \dots, x_{\kappa}), \dots, \sigma(x_{\kappa (n-1)+1}, \dots, x_{\kappa n})\|_{M}$$
for every $x \in \R^{\kappa n}$ and for every $\sigma \in \SO(\kappa)$. We work exclusively with geometric objects in $\R^{\kappa n}$ that are $\kappa$-balanced. For the purpose of clarity and the ease of notation, we denote by $\K^n$ the space $\R^{\kappa n}$ with the additional property that all geometric objects in this space satisfy the above invariance.

We call a set in $\K^n$ a \text{\it convex body} if it is a compact $\kappa$-balanced convex set in $\R^{\kappa n}$ with non-empty interior. Recall that a compact subset $K$ of $\R^n$ containing the origin as an interior point is called a \textit{star body} if every line through the origin crosses the boundary in exactly two points different from the origin. Its \textit{radial function} is defined by
$$ \rho_K(x) := \max \{ a \geq 0 \: ax\in K\} \, ,$$
with $x\in \R^n$. For $x\in\sn$, $\rho_K(x)=\|x\|_K^{-1}$, is the Euclidean distance from the origin to the boundary of $K$ in the direction $x$. The set of $\kappa$-balanced star bodies in $\R^{\kappa n}$ forms the class of \text{\it star bodies} in $\K^n$.

Now we introduce the notion of a hyperplane in $\K^n$. For a vector $y\in \R^{\kappa n}$, denote by $|y|_2$ the Euclidean length of $y$. Denote by $e_i$ the vector in $\R^{\kappa n}$ with the $i$-th coordinate equal to one and all other coordinates equal to zero.

\begin{lemma}\label{lemma_hyperplane}
Fix $\kappa\geq 2$ and a vector $x\in \R^{\kappa n}$. The set $\{R_{\sigma} (x) \: \sigma\in\SO(\kappa) \}$ is a $(\kappa-1)$-dimensional sphere of radius $|x|_2$ and its span, denoted by $H_x^{\perp}$, is a $\kappa$-dimensional subspace of $\R^{\kappa n}$. Moreover, the orthogonal complement of $H_x^{\perp}$, denoted by $H_{x}$, is a $\kappa$-balanced $(\kappa n -\kappa)$-dimensional subspace of $\R^{\kappa n}$.
\end{lemma}

\begin{proof}
It is enough to establish the claim for any convenient vector $x$. Choose $x=(0,\dots, 0, 1)$, then $\{R_{\sigma} (x) \: \sigma\in\SO(\kappa) \}$ is the unit sphere of the $\kappa$-dimensional subspace spanned by the vectors $e_{\kappa (n-1)+1}, \dots, e_{\kappa n}$.
\end{proof}

For an element $\xi \in \s^{\kappa n - 1}$, we call $H_{\xi}$ the \text{\it hyperplane in $\K^n$} determined by the vector $\xi$. 

\begin{df}
Let $D$ and $L$ be star bodies in $\K^n$. We call $D$ the intersection body of $L$ in $\K^n$ and denote it by $D=I_{\K}(L)$ if for every $\xi \in \s^{\kappa n - 1}$
\begin{equation}\label{eq_IntersectionBody}
		|D\cap H_{\xi}^{\perp}| = |L\cap H_{\xi}| \, .
\end{equation}
\end{df}

Observe that for a $\kappa$-balanced star body $D$, the set $D\cap H_{\xi}^{\perp}$ is a bounded $\kappa$-dimensional $\kappa$-balanced subset of $\R^{\kappa n}$ and hence it is a $(\kappa-1)$-dimensional ball of radius $\|\xi\|^{-1}_D$. Thus, by the polar formula for the volume, equation (\ref{eq_IntersectionBody}) becomes
\begin{equation}\label{eq_IntersectionBodyPolar}
 \frac{|S^{\kappa-1}|}{\kappa} \|\xi\|^{-\kappa}_{I_{\K}(L)} = |L\cap H_{\xi}| \, ,
\end{equation}
where $|S^{\kappa-1}|$ stands for the surface area of the unit ball in $\R^{\kappa}$. 

\section{The Radon and Fourier Transforms of $\kappa$-invariant Functions}\label{SectionRadonFourierTransforms}

We call a function $f$ on $\R^{\kappa n}$ \text{\it $\kappa$-invariant} if $f(x)=f(R_{\sigma} x)$ for every $x\in \R^{\kappa n}$ and $\sigma \in \SO(\kappa)$, and denote the space of continuous $\kappa$-invariant real-valued functions on the unit sphere by $C_{\kappa}(\s^{\kappa n - 1})$. 
The \text{\it spherical Radon transform} \text{\it on $\K^n$}, denote it by $\RT^{\kappa}$, is an operator from $C_{\kappa}(\s^{\kappa n - 1})$ to itself, defined by 
$$ \RT^{\kappa}f(\xi) = \int_{\s^{\kappa n - 1} \cap H_{\xi}} f(x) dx \, . $$
The polar formula for the volume yields
\begin{equation}\label{eq_var}
 |L\cap  H_{\xi} | = \frac{1}{\kappa n - \kappa} \RT^{\kappa} (\|\cdot\|_L^{-\kappa n + \kappa})(\xi) \,
\end{equation}
for any star body $L$ in $\K^n$ and $\xi \in \s^{\kappa n - 1}$. Moreover, condition (\ref{eq_IntersectionBody}) becomes
\begin{equation}\label{eq_IntersectionBodyRadon}
		  \|\xi\|^{-\kappa}_{I_{\K}(L)} = \frac{1}{(n -1) |S^{\kappa-1}|} \RT^{\kappa} (\|\cdot\|_L^{-\kappa n + \kappa})(\xi) \, .
\end{equation}
We conclude that a star body $D$ in $\K^n$ is the intersection body of a star body if and only if the function $\|\xi\|^{-\kappa}_D$ is the spherical Radon transform on $\K^n$ of a positive $\kappa$-invariant continuous function on $\s^{\kappa n - 1}$.

We will generalize several classical facts, connecting the Radon and Fourier transforms. We start by recalling the relevant concepts and facts in $\R^n$.

One of the main tools used in this paper is the Fourier transform of distributions, see \cite{GelfandShilov1964} for details. Denote by $\sw(\Rn)$ the \textit{Schwartz space} of rapidly decreasing infinitely differentiable functions on $\Rn$, also referred to as \textit{test functions}, and by $\sw'(\Rn)$ the space of \textit{distributions} on $\Rn$. The Fourier transform $\hat{f}$ of a distribution $f$ is defined by $\langle \hat{f}, \varphi \rangle = \left\langle f, \hat{\varphi} \right\rangle$ for every test function $\varphi$. For an even test function $\varphi$, the Fourier transform is self-invertible up to a constant factor: $(\varphi^{\wedge})^{\wedge}=(2\pi)^n\varphi$. A distribution $f$ on $\Rn$ is \textit{even homogeneous of degree $p\in\R$}, if 
$$ \left\langle f(x), \varphi\left(\frac{x}{\alpha}\right) \right\rangle = |\alpha|^{n+p} \left\langle f, \varphi  \right\rangle $$
for every test function $\varphi$ and every $\alpha\in\R, \alpha\neq 0$. The Fourier transform of an even homogeneous distribution of degree $p$ is an even homogeneous distribution of degree $-n-p$. We call a distribution $f$ \textit{positive definite} if its Fourier transform is a positive distribution, i.e. $\langle \hat{f}, \varphi \rangle \geq 0$ for every non-negative test function $\varphi$. A measure $\mu$ is \textit{tempered} if for some $\beta > 0$
$$ \int_{\R^n} (1+|x|_2)^{-\beta} d\mu(x) < \infty \, ,$$
where $|\cdot|_2$ denotes the Euclidean norm on $\R^n$. A distribution is positive definite if and only if it is the Fourier transform of a tempered measure on $\R^n$, see \cite{GelfandVilenkin64}, p.152. Let $D$ be an origin-symmetric star body in $\R^n$. For $0<p<n$, the function $\|\cdot\|_D^{-p}$ is locally integrable on $\R^n$, and represents an even homogeneous distribution of degree $-p$, see \cite{Koldobsky2005}, Lemma 2.1. If $\|\cdot\|_D^{-p}$ is also positive definite, then its Fourier transform is a homogeneous distribution of degree $-n+p$. Also, there exists a measure $\mu$ on the sphere so that
\begin{equation}\label{eq_posdefmsre}
\int_{\R^n} \|x\|_D^{-p} \varphi(x) dx = \int_{\s^{n-1}} \left( \int_0^{\infty} t^{p-1} \hat{\varphi}(t\xi) dt \right) d\mu(\xi) \, , 
\end{equation}
for every test function $\varphi$, see \cite{Koldobsky2005}, Corollary 2.26 (i).

Let $f$ be an even continuous function on $\s^{n-1}$ and let $p$ be a non-zero real number. We extend $f$ to an even homogeneous function on $\R^n$ of degree $p$ in the usual way as follows. Let $x\in\R^n$, then $x=r \theta$ with $r=|x|_2$ and $\theta = x/|x|_2$. We write
$$ f\cdot r^p(x)=f(\theta) r^p \, .$$  
It was shown in \cite{Koldobsky2005}, Lemma 3.16, that for an infinitely-smooth function $f$ on $\s^{n-1}$ and $-n<p<0$, the Fourier transform of $f \cdot r^{-p}$ is an infinitely-smooth function on $\Rn \setminus \{0\}$, homogeneous of degree $-n+p$. 

We shall often use Parseval's formula on the sphere:
\begin{lemma}\textnormal{(\cite{Koldobsky2005}, Lemma 3.22)}\label{lemma_ParsevalOnTheSphere}
Let $f$ and $g$ be even infinitely-smooth functions on $\s^{n-1}$ and let $0<p<n$. Then 
$$ \int_{\sn} (f\cdot r^{-p})^{\wedge}(\theta) (g\cdot r^{-n+p})^{\wedge}(\theta) d\theta = (2\pi)^n \int_{\sn} f(\theta) g(\theta) d\theta \, . $$
\end{lemma} 

\noindent
Another basic fact from Fourier analysis is the following.
\begin{lemma}\textnormal{(\cite{Koldobsky2005}, Lemma 3.24)}\label{lemma_ftf}
Let $0<k<n$, and let $\varphi\in \sw(\R^n)$ be an even test function. Then for any $(n-k)$-dimensional subspace $H$ of $\R^n$
$$ (2\pi)^k \int_{H} \varphi(x) dx = \int_{H^{\perp}} \hat{\varphi}(x) dx \, . $$
\end{lemma} 

The spherical version of the above lemma allows to express the volume of lower-dimensional sections of an origin-symmetric star body in $\R^n$ in Fourier analytic terms.   
\begin{lemma}\textnormal{(\cite{Koldobsky2005}, Lemma 3.25)}\label{lemma_ftlds}
Let $0<k<n$, and let $\varphi$ be an even infinitely-smooth function on $\s^{n-1}$. Then for any $(n-k)$-dimensional subspace $H$ of $\R^n$
$$ (2\pi)^k \int_{\s^{n-1} \cap H} \varphi(x) dx = \int_{\s^{n-1} \cap H^{\perp}} (\varphi \cdot r^{-n+k})^{\wedge}(x) dx \, . $$
\end{lemma} 

The $\kappa$-invariance of a function translates into a certain invariance of its Fourier transform.

\begin{lemma}\label{lemma_ftc}
Suppose that $f$ is an even infinitely-smooth $\kappa$-invariant function on $\s^{\kappa n -1}$. Then for any $0<p<\kappa n$ and any $\xi \in \s^{\kappa n-1}$ the Fourier transform of the distribution $f\cdot r^{-p}$ is a constant function on $\s^{\kappa n-1}\cap H_{\xi}^{\perp}$.
\end{lemma}

\begin{proof}
The Fourier transform of $f\cdot r^{-p}$ is a continuous function outside of the origin in $\R^{\kappa n}$ by Lemma 3.16 in \cite{Koldobsky2005}. Since the function $f$ is $\kappa$-invariant, by the connection between the Fourier transform of distributions and linear transformations, the Fourier transform of $f\cdot r^{-p}$ is also $\kappa$-invariant. The set $\s^{\kappa n-1}\cap H_{\xi}^{\perp}$ is the unit sphere of the $\kappa$-balanced $\kappa$-dimensional subspace $H_{\xi}^{\perp}$, see Lemma \ref{lemma_hyperplane}. Hence every vector in $\s^{\kappa n-1}\cap H_{\xi}^{\perp}$ is the image of $\xi$ under one of the $\kappa$-tuple-wise rotations in $\SO(\kappa n)$ and consequently the Fourier transform of $f\cdot r^{-p}$ is a constant function on $\s^{\kappa n-1}\cap H_{\xi}^{\perp}$.
\end{proof}

\begin{lemma}\label{lemma_rft}
Let $\varphi$ be an even infinitely-smooth $\kappa$-invariant function on $\s^{\kappa n - 1}$, then for $\xi \in \s^{\kappa n - 1}$
$$ \RT^{\kappa} \varphi (\xi) = \frac{|S^{\kappa-1}|}{(2 \pi)^{\kappa}} (\varphi \cdot r^{-\kappa n+\kappa})^{\wedge}(\xi)\, . $$
\end{lemma}

\begin{proof}
By Lemma \ref{lemma_ftlds}, we have
$$ \RT^{\kappa} \varphi (\xi) = \int_{\s^{\kappa n - 1} \cap H_{\xi}} \varphi(x) dx = \frac{1}{(2 \pi)^{\kappa}} \int_{\s^{\kappa n - 1} \cap H_{\xi}^{\perp}} (\varphi \cdot r^{-\kappa n+\kappa})^{\wedge}(x) dx \, .$$
Since $\varphi$ is $\kappa$-invariant, by Lemma \ref{lemma_ftc} the integrand on the right-hand side is a constant function on $\s^{\kappa n - 1} \cap H_{\xi}^{\perp}$, which itself is a $(\kappa-1)$-dimensional Euclidean unit sphere. Hence
$$ \int_{\s^{\kappa n - 1} \cap H_{\xi}^{\perp}} (\varphi \cdot r^{-\kappa n+\kappa})^{\wedge}(x) dx = |S^{\kappa-1}| (\varphi \cdot r^{-\kappa n+\kappa})^{\wedge}(\xi)\, .$$
\end{proof}

The smoothness assumption in the above lemma can be removed. It is an analog of Lemma 3.7 in \cite{Koldobsky2005}, see also Lemma 4 in \cite{MR3106735}. Beforehand we need the following fact.

\begin{lemma}\label{lemma_rtsd}
The spherical Radon transform on $\K^n$ is self-dual, i.e. for any even continuous $\kappa$-invariant functions $f, g$ on $\s^{\kappa n - 1}$
$$ \int_{\s^{\kappa n - 1}} \RT^{\kappa} f(\xi) g(\xi) d\xi = \int_{\s^{\kappa n - 1}} f(\theta) \RT^{\kappa} g(\theta) d\theta \,. $$
\end{lemma}

\begin{proof}
We can assume that functions $f, g$ are infinitely-smooth. The Fourier transform of the homogeneous extension of $g$ of degree $-\kappa n + \kappa$ is an infinitely-smooth $\kappa$-invariant homogeneous function of degree $-\kappa$ on $\R^{\kappa n} \setminus \{0\}$, so for some infinitely-smooth $\kappa$-invariant function $h$ on $\s^{\kappa n - 1}$
$$ (g \cdot r^{-\kappa n + \kappa})^{\wedge} = (2\pi)^{\kappa n} h \cdot r^{-\kappa} \, .$$
Using Lemma \ref{lemma_rft} and spherical Parseval's formula, we now compute
\begin{align*}
 \int_{\s^{\kappa n - 1}} \RT^{\kappa} f(\xi) g(\xi) d\xi 
 		&= \frac{|S^{\kappa-1}|}{(2 \pi)^{\kappa}} \int_{\s^{\kappa n - 1}}  (f \cdot r^{-\kappa n+\kappa})^{\wedge}(\xi) (g\cdot r^{-\kappa n + \kappa})(\xi) d\xi \\
 		&= \frac{|S^{\kappa-1}|}{(2 \pi)^{\kappa}} \int_{\s^{\kappa n - 1}}  (f \cdot r^{-\kappa n+\kappa})^{\wedge}(\xi) (h \cdot r^{-\kappa})^{\wedge}(\xi) d\xi \\
 		&= \frac{|S^{\kappa-1}| (2\pi)^{\kappa n}}{(2 \pi)^{\kappa}} \int_{\s^{\kappa n - 1}}  f(\theta) (h \cdot r^{-\kappa})(\theta) d\theta \\
 		&= \frac{|S^{\kappa-1}|}{(2 \pi)^{\kappa}} \int_{\s^{\kappa n - 1}}  f(\theta) (g \cdot r^{-\kappa n + \kappa})^{\wedge}(\theta) d\theta \\
 		&= \int_{\s^{\kappa n - 1}}  f(\theta) \RT^{\kappa} g(\theta) d\theta \,.
\end{align*}  
\end{proof}

We say that a distribution $f$ on $\R^{\kappa n}$ is \text{\it $\kappa$-invariant} if 
$\left\langle f, \varphi(R_{\sigma} \cdot) \right\rangle=\left\langle f, \varphi \right\rangle$ for every test function $\varphi$ and for every $\sigma \in \SO(\kappa)$. Note that if two $\kappa$-invariant distributions coincide on the set of $\kappa$-invariant test functions, then they are equal. Indeed, let $f$ be one such distribution and let $\varphi$ be any test function. For $x\in \R^n$, set $\varphi_0(x) = \int_{\SO(\kappa)} \varphi(R_{\sigma} x) \, d\sigma$, where $d\sigma$ stands for the Haar probability measure on $\SO(\kappa)$. Then $\varphi_0$ is a $\kappa$-invariant test function and 
$$ \left\langle f, \varphi \right\rangle = \left\langle f, \varphi(R_{\sigma} \cdot) \right\rangle = \left\langle f, \int_{\SO(\kappa)} \varphi(R_{\sigma} \cdot) d\sigma \right\rangle = \left\langle f, \varphi_0 \right\rangle \, .$$

\begin{lemma}\label{lemma_rfts}
Let $f$ be an even continuous $\kappa$-invariant function on $\s^{\kappa n - 1}$, then for $\xi \in \s^{\kappa n - 1}$
$$ \RT^{\kappa} f (\xi) = \frac{|S^{\kappa-1}|}{(2 \pi)^{\kappa}} (f \cdot r^{-\kappa n+\kappa})^{\wedge}(\xi)\, .$$
\end{lemma}

\begin{proof}
Let $\varphi$ be any $\kappa$-invariant test function, then
$$ \int_{H_{\xi}^{\perp}} \hat{\varphi}(x) dx = \int_{\s^{\kappa n - 1} \cap H_{\xi}^{\perp}} \int_0^{\infty} \hat{\varphi}(r \theta) r^{\kappa-1} dr \, d\theta = |S^{\kappa-1}| \int_0^{\infty} \hat{\varphi}(r \xi) r^{\kappa-1} dr \, .$$
Using this observation, we compute
\begin{align*}
	\left\langle (f\cdot r^{-\kappa n + \kappa})^{\wedge}, \varphi \right\rangle 
		&= \int_{\s^{\kappa n - 1}} f(\theta) \int_0^{\infty} \hat{\varphi}(r \theta) r^{\kappa-1} dr \, d\theta \\
		&= \frac{1}{|S^{\kappa-1}|} \int_{\s^{\kappa n - 1}} f(\theta) \int_{H_{\theta}^{\perp}} \hat{\varphi}(x) dx \, d\theta \\
		\intertext{by Lemma \ref{lemma_ftf}}
		&= \frac{(2\pi)^{\kappa}}{|S^{\kappa-1}|} \int_{\s^{\kappa n - 1}} f(\theta) \int_{H_{\theta}} \varphi(x) dx \, d\theta \\
		&= \frac{(2\pi)^{\kappa}}{|S^{\kappa-1}|} \int_{\s^{\kappa n - 1}} f(\theta) \int_{\s^{\kappa n - 1} \cap H_{\theta}} \int_0^{\infty} \varphi(rx) r^{\kappa n -\kappa -1} dr \, dx \, d\theta \\
		&= \frac{(2\pi)^{\kappa}}{|S^{\kappa-1}|} \int_{\s^{\kappa n - 1}} f(\theta) \RT^{\kappa}\left( \int_0^{\infty} \varphi(r\cdot) r^{\kappa n -\kappa -1} dr \right)(\theta) d\theta\\
		\intertext{by Lemma \ref{lemma_rtsd}}
		&= \frac{(2\pi)^{\kappa}}{|S^{\kappa-1}|} \int_{\s^{\kappa n - 1}} \RT^{\kappa} f(\theta) \int_0^{\infty} \varphi(r \theta) r^{\kappa n -\kappa -1} dr \, d\theta \\
		&= \frac{(2\pi)^{\kappa}}{|S^{\kappa-1}|}  \left\langle |x|_2^{-\kappa} \RT^{\kappa} f(x/|x|_2), \varphi \right\rangle \, .
\end{align*}
This shows that $\kappa$-invariant distributions $(f\cdot r^{-\kappa n + \kappa})^{\wedge}$ and $\frac{(2\pi)^{\kappa}}{|S^{\kappa-1}|} |x|_2^{-\kappa} \RT^{\kappa} f(\frac{x}{|x|_2}) $ coincide on the set of $\kappa$-invariant test functions and are therefore equal.
\end{proof}

The above lemma allows to express the volume of sections of star bodies as the Fourier transform of a certain function. The real version of this fact was proved in \cite{K1} and the complex version was proved in \cite{KoldobskyKonigZymonopoulou2008}, \cite{MR3106735}.

\begin{theorem}\label{thm_vos}
For any origin-symmetric star body $D$ in $\K^n$ and for any unit vector $\xi\in \R^{\kappa n}$, we have
$$ |D\cap  H_{\xi} | =  \frac{|S^{\kappa-1}|}{(2 \pi)^{\kappa} (\kappa n - \kappa)} (\|\cdot\|_D^{-\kappa n + \kappa})^{\wedge}(\xi) \, , $$
where $H_{\xi}$ is the hyperplane in $\K^n$ determined by $\xi$, see Section \ref{SectionIntersectionBodiesOfStarBodies} for the definition.
\end{theorem}

\begin{proof}
By (\ref{eq_var}) and Lemma \ref{lemma_rfts}, we obtain
$$ |D\cap  H_{\xi} |  = \frac{1}{\kappa n - \kappa} \RT^{\kappa} (\|\cdot\|_D^{-\kappa n + \kappa})(\xi) =  \frac{|S^{\kappa-1}|}{(2 \pi)^{\kappa} (\kappa n - \kappa)} (\|\cdot\|_D^{-\kappa n + \kappa})^{\wedge}(\xi) \, .$$
\end{proof}

Theorem \ref{thm_vos} provides a version of the Funk-Minkowski uniqueness theorem, see \cite{MR2251886}, Th. 7.2.3.

\begin{co}
Let $K, L$ be origin-symmetric star bodies in $\K^n$. If for every direction $\xi\in \s^{\kappa n-1}$
$$ |K\cap  H_{\xi} | = |L \cap  H_{\xi} | \, ,$$
then $K=L$.
\end{co}

\begin{proof}
By Theorem \ref{thm_vos}, the hypothesis of the corollary implies that homogeneous of degree $-\kappa$ continuous functions on $\R^{\kappa n}\setminus \{0\}$, $(\|\cdot\|_K^{-\kappa n + \kappa})^{\wedge}$ and $(\|\cdot\|_L^{-\kappa n + \kappa})^{\wedge}$ coincide on the sphere $\s^{\kappa n -1}$. Thus they coincide as distributions on the whole $\R^{\kappa n}$. The result follows by the uniqueness theorem for the Fourier transform of distributions. 
\end{proof}

\section{Intersection Bodies in $\K^n$}\label{SectionIntersectionBodies}
Intersection bodies of star bodies in $\K^n$ were introduced in Section \ref{SectionIntersectionBodiesOfStarBodies}. Now we define a more general class of intersection bodies by extending the equality (\ref{eq_IntersectionBodyRadon}) to measures, as it was done in \cite{GoodeyLutwakWeil1996} for the real case and in \cite{MR3106735} for the complex case. A finite Borel measure $\mu$ on the sphere $\s^{\kappa n -1}$ is called \text{\it $\kappa$-invariant} if for any continuous function $f$ on the sphere $\s^{\kappa n -1}$ and for any $\sigma \in \SO(\kappa)$
$$ \int_{\s^{\kappa n -1}} f(x) d\mu(x) = \int_{\s^{\kappa n -1}} f(R_{\sigma} x) d\mu(x) \, .$$
The spherical Radon transform on $\K^n$ of an $\kappa$-invariant measure $\mu$ on the sphere $\s^{\kappa n -1}$ is defined as a functional $\RT^{\kappa} \mu$ on the space $C_{\kappa}(\s^{\kappa n - 1})$ by
$$ (\RT^{\kappa} \mu, f) = \int_{\s^{\kappa n -1}} \RT^{\kappa} f(x) d\mu(x)  \, .$$
Surely, the spherical Radon transform on $\K^n$ of a finite $\kappa$-invariant Borel measure $\mu$ on $\s^{\kappa n -1}$, $\RT^{\kappa} \mu$, is again a finite $\kappa$-invariant Borel measure on $\s^{\kappa n -1}$. From the self-duality of the spherical Radon transform on $\K^n$, Lemma \ref{lemma_rtsd}, it follows that if the measure $\mu$ has a continuous density $g$, then the measure $\RT^{\kappa} \mu$ has the density $\RT^{\kappa} g$.

\begin{df}\label{df_ib}
An origin-symmetric star body $K$ in $\K^n$ is called an intersection body in $\K^n$ if there exists a finite $\kappa$-invariant Borel measure $\mu$ on the sphere $\s^{\kappa n -1}$ so that $\|\cdot\|_K^{-\kappa}$ and $\RT^{\kappa} \mu$ are equal as functionals on $C_{\kappa}(\s^{\kappa n - 1})$; that is, if for any $f\in C_{\kappa}(\s^{\kappa n - 1})$
\begin{equation*}
	\int_{\s^{\kappa n -1}} \|x\|_K^{-\kappa} f(x) dx = \int_{\s^{\kappa n -1}} \RT^{\kappa} f(x) d\mu(x) \, .
\end{equation*}
\end{df}

It follows from the self-duality of the spherical Radon transform on $\K^n$ and equation (\ref{eq_IntersectionBodyRadon}), that every intersection body of a star body in $\K^n$ is an intersection body in $\K^n$ in the sense of Definition \ref{df_ib}. 

It was shown in \cite{K2} that intersection bodies in $\R^n$ admit the following Fourier analytic characterization: an origin-symmetric star body $K$ in $\R^n$ is an intersection body if and only if the function $\|\cdot\|_K^{-1}$ represents a positive definite distribution. Intersection bodies in $\K^n$ allow for a similar characterization. It is easy to see this for intersection bodies of star bodies in $\K^n$. By Theorem \ref{thm_vos} we have:
$$ 	\|\xi\|^{-\kappa}_{I_{\K}(L)} = \frac{\kappa}{|S^{\kappa-1}|} |L \cap H_{\xi}| = \frac{1}{(2\pi)^{\kappa} (n-1)} (\|\cdot\|_L^{-\kappa n + \kappa})^{\wedge}(\xi) \, .$$
Both sides are even homogeneous functions of degree $-\kappa$ and agree on $\s^{\kappa n-1}$, so they are equal as distributions on $\R^{\kappa n}$. Since the Fourier transform of even distributions is self-invertible up to a constant factor, we get
\begin{equation}\label{eq_ibpd}
\left( \|\cdot\|^{-\kappa}_{I_{\K}(L)} \right)^{\wedge} = \frac{(2\pi)^{\kappa n}}{(2\pi)^{\kappa} (n-1)} \|\cdot\|_L^{-\kappa n + \kappa} > 0 \, .
\end{equation}
Thus $\|\cdot\|^{-\kappa}_{I_{\K}(L)}$ is positive definite. Furthermore, if the Fourier transform of $\|\cdot\|^{-\kappa}_K$ is an even strictly positive $\kappa$-invariant function on the sphere, then using equation (\ref{eq_ibpd}) we can construct a star body $L$ in $\K^n$ so that $K=I_{\K}(L)$. Next we show that this Fourier analytic characterization holds for arbitrary intersection bodies in $\K^n$.

\begin{theorem}\label{th_kibpdd}
An origin-symmetric star body $D$ in $\K^n$ is an intersection body in $\K^n$ if and only if $\|\cdot\|^{-\kappa}_D$ represents a positive definite distribution on $\R^{\kappa n}$. 
\end{theorem}

\begin{proof}
Suppose that $D$ is an intersection body in $\K^n$ with the corresponding measure $\mu$. It is enough to show $\left\langle (\|\cdot\|^{-\kappa}_D)^{\wedge}, \varphi \right\rangle \geq 0$ for every even $\kappa$-invariant non-negative test function $\varphi$. We compute
\begin{align*}
	\left\langle (\|\cdot\|^{-\kappa}_D)^{\wedge}, \varphi \right\rangle 	
		&= \int_{\R^{\kappa n}} \|x\|^{-\kappa}_D \hat{\varphi}(x) dx \\
		&= \int_{\s^{\kappa n -1}} \|\theta\|^{-\kappa}_D \left( \int_0^{\infty} \hat{\varphi}(r\theta) r^{\kappa n - \kappa -1} dr \right)  d\theta \\
		\intertext{by Definition \ref{df_ib}}
		&= \int_{\s^{\kappa n -1}} \RT^{\kappa} \left( \int_0^{\infty} \hat{\varphi}(r\cdot) r^{\kappa n - \kappa -1} dr \right)(\theta)  d\mu(\theta) \\
		&= \int_{\s^{\kappa n -1}} \int_{\s^{\kappa n -1}\cap H_{\theta}} \int_0^{\infty} \hat{\varphi}(rx) r^{\kappa n - \kappa -1} dr \,  dx \, d\mu(\theta) \\
		&= \int_{\s^{\kappa n -1}} \left( \int_{H_{\theta}} \hat{\varphi}(x) dx \right)  d\mu(\theta) \\
		\intertext{by Lemma \ref{lemma_ftf}}
		&= (2\pi)^{\kappa n - \kappa} \int_{\s^{\kappa n -1}} \left( \int_{H_{\theta}^{\perp}} \varphi(x) dx \right)  d\mu(\theta) \geq 0 \, .
\end{align*}

Conversely, suppose that $\|\cdot\|^{-\kappa}_D$ is a positive definite distribution, then there exists a finite Borel measure $\mu$ on $\s^{\kappa n -1}$ so that for every test function $\varphi$
\begin{equation}\label{eq_heqpdd}
 \int_{\R^{\kappa n}} \|x\|_D^{-\kappa} \varphi(x) dx = \int_{\s^{\kappa n-1}} \left( \int_0^{\infty} t^{\kappa-1} \hat{\varphi}(t\xi) dt \right) d\mu(\xi) \, , 
\end{equation}
see (\ref{eq_posdefmsre}). Since the body $D$ is $\kappa$-balanced, we can assume that the measure $\mu$ is $\kappa$-invariant. Indeed, define a $\kappa$-invariant measure $\tilde{\mu}$ on the sphere $\s^{\kappa n - 1}$ by $ \tilde{\mu}(E)=\int_{\SO(\kappa)} \mu(R_{\sigma}E) d\sigma$, with $E\subset \s^{\kappa n - 1}$ and $d\sigma$ the Haar probability measure on $\SO(\kappa)$. Then for a $\kappa$-invariant test function $\varphi$, we have
\begin{align*}
	 \int_{\s^{\kappa n-1}} \left( \int_0^{\infty} t^{\kappa-1} \hat{\varphi}(t\xi) dt \right) d\tilde{\mu}(\xi) &= \int_{\SO(\kappa)}  \int_{\R^{\kappa n}} \|R_{\sigma} x\|_D^{-\kappa} \varphi(R_{\sigma} x) dx \, d\sigma \\
	 &= \int_{\SO(\kappa)}  \int_{\R^{\kappa n}} \| x\|_D^{-\kappa} \varphi(x) dx \, d\sigma \\
	 &= \int_{\s^{\kappa n-1}} \left( \int_0^{\infty} t^{\kappa-1} \hat{\varphi}(t\xi) dt \right) d\mu(\xi) \, .
\end{align*}
Thus we showed that there exists a $\kappa$-invariant measure $\tilde{\mu}$ that coincides with the measure $\mu$ on the set of $\kappa$-invariant test functions and hence on the set of all test functions. To see this, let $\psi$ be any test function and $\psi_0(x) = \int_{\SO(\kappa)} \psi(R_{\sigma}x) d\sigma$ as on p. 8. Note that $(\psi_0)^{\wedge} = (\hat\psi)_0$. By the argument presented before Lemma \ref{lemma_rfts} and above observations, we have
\begin{align*}
	 \int_{\s^{\kappa n-1}} \left( \int_0^{\infty} t^{\kappa-1} \hat{\psi}(t\xi) dt \right) d\tilde{\mu}(\xi) 
	 &= \int_{\s^{\kappa n-1}} \left( \int_0^{\infty} t^{\kappa-1} (\hat{\psi})_0(t\xi) dt \right) d\tilde{\mu}(\xi) \\
	 &= \int_{\s^{\kappa n-1}} \left( \int_0^{\infty} t^{\kappa-1} (\psi_0)^{\wedge}(t\xi) dt \right) d\mu(\xi) \\
	 &= \int_{\R^{\kappa n}} \|x\|_D^{-\kappa} \psi_0(x) dx \\
	 &=  \int_{\SO(\kappa)} \left( \int_{\R^{\kappa n}} \|R_{\sigma} x\|_D^{-\kappa} \psi(R_{\sigma} x) dx \right) d\sigma \\
	 &= \int_{\R^{\kappa n}} \|x\|_D^{-\kappa} \psi(x) dx \\
	 &= \int_{\s^{\kappa n-1}} \left( \int_0^{\infty} t^{\kappa-1} \hat{\psi}(t\xi) dt \right) d\mu(\xi) \, .
\end{align*}

Recall, from the proof of Lemma \ref{lemma_rfts}, that for a $\kappa$-invariant test function $\varphi$
$$ \int_{H_{\xi}^{\perp}} \hat{\varphi}(x) dx = |S^{\kappa-1}| \int_0^{\infty} \hat{\varphi}(r \xi) r^{\kappa-1} dr \, .$$
Thus for even $\kappa$-invariant test functions, the right-hand side of equation (\ref{eq_heqpdd}) can be written as
\begin{align*}
	\frac{1}{|S^{\kappa-1}|}\int_{\s^{\kappa n-1}} & \left( \int_{H_{\xi}^{\perp}} \hat{\varphi}(x) dx \right)  d\mu(\xi) 
		= \frac{(2\pi)^{\kappa}}{|S^{\kappa-1}|} \int_{\s^{\kappa n-1}} \left( \int_{H_{\xi}} \varphi(x) dx \right) d\mu(\xi) \, , \\
	\intertext{where we used Lemma \ref{lemma_ftf}, and now, writing the interior integral in polar coordinates, we obtain}
		&= \frac{(2\pi)^{\kappa}}{|S^{\kappa-1}|} \int_{\s^{\kappa n-1}} \RT^{\kappa} \left( \int_0^{\infty} \varphi(r\cdot) r^{\kappa n-\kappa -1} dr \right)(\xi) d\mu(\xi) \, .
\end{align*}
Writing the left-hand side in equation (\ref{eq_heqpdd}) in polar coordinates, we obtain that for any even $\kappa$-invariant test function $\varphi$
\begin{align}
	\int_{\s^{\kappa n-1}} \|\theta\|_D^{-\kappa} & \left( \int_0^{\infty} \varphi(r\theta) r^{\kappa n-\kappa -1} dr \right)
	d\theta \nonumber \\
	&= \frac{(2\pi)^{\kappa}}{|S^{\kappa-1}|} \int_{\s^{\kappa n-1}} \RT^{\kappa} \left( \int_0^{\infty} \varphi(r\cdot) r^{\kappa n-\kappa -1} dr \right)\!(\xi) \, d\mu(\xi) \, . \label{eq_hepcpd}
\end{align}
Let $u$ be some non-negative test function on $\R$ and let $v$ be an arbitrary infinitely-smooth even $\kappa$-invariant function on $\s^{\kappa n -1}$. For $x\in \R^{\kappa n}$, set $\varphi(x)=u(r)v(\theta)$, where $x=r\theta$ with $r\in[0,\infty)$ and $\theta\in\s^{\kappa n -1}$. Evaluating equation (\ref{eq_hepcpd}) for such test functions $\varphi$, yields
$$ \int_{\s^{\kappa n-1}} \|\theta\|_D^{-\kappa} v(\theta)	d\theta= \frac{(2\pi)^{\kappa}}{|S^{\kappa-1}|} \int_{\s^{\kappa n-1}} \RT^{\kappa}v(\xi) d\mu(\xi) \, .$$
Since infinitely-smooth functions on the sphere are dense in the space of continuous functions on the sphere, the latter equation holds for all $v\in C_{\kappa}(\s^{\kappa n - 1})$, which implies that $D$ is an intersection body in $\K^n$.
\end{proof}

\section{Geometric Characterizations of intersection bodies in $\K^n$}\label{section_cibinkn}
Intersection bodies in $\K^n$ are related to two generalizations of real intersection bodies. Consequently they inherit many of their properties. 

One generalization, $k$-intersection bodies, was introduced by A. Koldobsky in \cite{K2, Koldobsky2000} as follows. Let $M, L$ be star bodies in $\R^n$ and let $k$ be an integer, $0<k<n$. We say that $M$ is a \text{\it $k$-intersection body of $L$} if for every $(n-k)$-dimensional subspace $H$ of $\R^n$
$$ |M\cap H^{\perp}| = |L\cap H| \, .$$ 
A more general class of $k$-intersection bodies was defined in \cite{Koldobsky2000} as follows.
\begin{df}
Let $0<k<n$. We say that an origin-symmetric star body $M$ in $\R^n$ is a $k$-intersection body if there exists a measure $\mu$ on $\s^{n-1}$ such that for every test function $\varphi$ in $\R^n$
\begin{equation*}
	\int_{\R^n} \|x\|_M^{-k} \varphi(x) dx = \int_{\s^{n-1}} \left( \int_0^{\infty} t^{k-1} \hat{\varphi}(t\xi) dt \right) d\mu(\xi) \,.
\end{equation*} 
\end{df} 
Equivalently, $k$-intersection bodies can be viewed as limits in the radial metric of $k$-intersection bodies of star bodies, see \cite{milman2006, rubin2008}. They are related to a certain generalization of the Busemann-Petty problem in the same way as intersection bodies are related to the original problem, see Section 5.2 in \cite{Koldobsky2005}. 

An origin-symmetric star body $K$ in $\R^n$ is a $k$-intersection body if and only if $\|\cdot\|_K^{-k}$ represents a positive definite distribution, see \cite{Koldobsky2000}. Thus Theorem \ref{th_kibpdd} implies,
\begin{co}\label{co_ibkib}
An origin-symmetric star body in $\K^n$ is an intersection body in $\K^n$ if and only if it is a $\kappa$-balanced $\kappa$-intersection body in  $\R^{\kappa n}$. 
\end{co}

Our next goal is   to determine when an origin-symmetric convex body in $\K^n$ is an intersection body. The answer to this question essentially follows from the so-called second derivative test, developed by A. Koldobsky, and Brunn's Theorem. For the convenience of the reader we recall related results and give necessary details. 

The concept of an embedding in $L_{p}$ was extended to negative values of $p$ in \cite{MR1703694}. Let $K$ be an origin-symmetric star body in $\R^n$. For $0<p<n$, we say \text{\it $(\R^n, \|\cdot\|_K)$ embeds in $L_{-p}$} if there exists a finite Borel measure $\mu$ on $\s^{n-1}$ so that for every test function $\phi$
$$ \int_{\R^n} \|x\|^{-p}_K \phi(x) dx = \int_{\s^{n-1}} \left( \int_{\R} |t|^{p-1} \hat{\phi}(t\theta) dt \right) d\mu(\theta) \, .$$
Note that an origin-symmetric star body $K$ in $\R^n$ is a $k$-intersection body if and only if the space $(\R^n, \|\cdot\|_K)$ embeds in $L_{-k}$. Here we only use embeddings in $L_{-p}$ to state results in a continuous form, for applications of this concept see Chapter 6 in \cite{Koldobsky2005}. Embeddings in $L_{-p}$ also admit a Fourier analytic characterization: The space $(\R^n, \|\cdot\|_K)$ embeds in $L_{-p}$ if and only if $\|\cdot\|_K^{-p}$ represents a positive definite distribution on $\R^n$.

It was proved in \cite{MR1694767} that every $n$-dimensional normed space embeds in $L_{-p}$ for each $p\in [n-3,n)$. In particular, every origin-symmetric convex body in $\R^n$ is a $k$-intersection body for $k=n-3, n-2, n-1$. 

Recall that for normed spaces $X, Y$ and $q\in \R$, $q\geq 1$, the $q$-sum $(X \oplus Y)_q$ of $X$ and $Y$ is defined as the space of pairs $\{(x,y) \: x\in X, y\in Y \}$ with the norm 
$$ \|(x,y)\| = (\|x\|_X^q + \|y\|_Y^q)^{1/q} .$$
The following result is an application of the second derivative test for embeddings in $L_{-p}$ and $k$-intersection bodies, which was first proved in \cite{MR1623670} and then generalized in \cite{Koldobsky2005}.

\begin{pr}\textnormal{(\cite{KoldobskyKonigZymonopoulou2008}, Proposition 3)}\label{pr_sdt}
Let $n\geq 3$ and let $Y$ be an $n$-dimensional normed space. For $q>2$, the $q$-sum of $\R$ and $Y$ does not embed in $L_{-p}$ with $0<p<n-2$. In particular, the unit ball of this direct sum is not a $k$-intersection body for any $1 \leq k < n-2$. 
\end{pr}
Next, we give an example of an origin-symmetric $\kappa$-balanced convex body in $\R^{\kappa n}$ that is not a $k$-intersection body for any $1 \leq k < \kappa(n-1)-2$. Denote by $\|\cdot\|_{q, \kappa}$, $q\geq 1$, the following norm on $\R^{\kappa n}$
$$ \|x\|_{q, \kappa} = \left( (x_1^2+\cdots +x_{\kappa}^2)^{q/2}+ \cdots +(x_{\kappa(n-1)+1}^2+\cdots +x_{\kappa n}^2)^{q/2} \right)^{1/q} \, ,$$
and by $B_{q, \kappa}$ the unit ball of the space $(\R^{\kappa n}, \|\cdot\|_{q, \kappa})$. Observe that the body $B_{q, \kappa}$ is $\kappa$-balanced.

\begin{pr}\label{pr_nkib}
For $q>2$, the space $(\R^{\kappa n}, \|\cdot\|_{q, \kappa} )$ does not embed in $L_{-p}$ for $0<p<\kappa(n-1)-2$. In particular, the body $B_{q, \kappa}$ is not a $k$-intersection body for any $1 \leq k < \kappa(n-1)-2$. 
\end{pr}

\begin{proof}
The space $(\R^{\kappa n}, \|\cdot\|_{q, \kappa})$ contains as a subspace the $q$-sum of $\R$ and a $\kappa (n-1)$-dimensional space $(\R^{\kappa (n-1)}, \|\cdot\|_{q, \kappa})$. By Proposition \ref{pr_sdt}, this $q$-sum does not embed in $L_{-p}$ with $0<p<\kappa(n-1)-2$. 
The larger space cannot embed in $L_{-p}$ with $0<p<\kappa(n-1)-2$ either, by a result of E. Milman, see Proposition 3.17 in \cite{milman2006}. 
\end{proof}

In Theorem \ref{th_kib} below, we will show that for an origin-symmetric $\kappa$-balanced convex body $K$ in $\R^{\kappa n}$ the space $(\R^{\kappa n}, \|\cdot\|_K)$ embeds in $L_{-p}$ for every $p>0$, $\kappa(n-1)-2 \leq p \leq \kappa (n-1)$. Its proof will invoke the so-called parallel section function. 

Let $0<k<n$ and let $H$ be an $(n-k)$-dimensional subspace of $\R^n$. Fix an orthonormal basis $e_1, \dots, e_k$ in the orthogonal subspace $H^{\perp}$. For a star body $K$ in $\Rn$, define the \textit{$(n-k)$-dimensional parallel section function $A_{K,H}$} as a function on $\R^k$ such that for $u\in \R^k$
\begin{align*}
	A_{K,H}(u)	&= \Vol_{n-k}(K\cap\{H+u_1 e_1 + \cdots + u_k e_k\}) \\
       				&= \int_{\{ x\in \Rn\: (x,e_1)=u_1, \dots , (x,e_k)=u_k \}} \chi (\|x\|_K) dx\, ,
\end{align*}
where $\chi$ is the indicator function of the interval $[0,1]$. 

For every $q\in \C$, the value of the distribution $|u|_2^{-q-k}$ on a test function $\phi \in \sw(\R^k)$ is defined in the usual way, see p. 71 in \cite{GelfandShilov1964}, and represents an entire function in $q\in \C$. If $K$ is infinitely smooth, the function $A_{K,H}$ is infinitely differentiable at the origin, see Lemma 2.4 in \cite{Koldobsky2005}, and the same regularization procedure can be applied to define the action of these distributions on the function $A_{K,H}$, see \cite{Koldobsky2005} and p. 356-359 in \cite{KoldobskyKonigZymonopoulou2008} for more details. 

The following proposition was first proved in \cite{Koldobsky2000}, we formulate it in the form as it appears in \cite{KoldobskyKonigZymonopoulou2008}:

\begin{pr}\textnormal{(\cite{KoldobskyKonigZymonopoulou2008}, Proposition 4)}\label{pr_psf}
Let $K$ be an infinitely smooth origin-symmetric star body in $\Rn$ and $0<k<n$. Then for every $(n-k)$-dimensional subspace $H$ of $\Rn$ and any $q \in \R, -k < q < n-k$,
\begin{equation}
\left\langle \frac{|u|_2^{-q-k}}{\Gamma(-q/2)}, A_{K,H}(u) \right\rangle = \frac{2^{-q-k} \pi^{-k/2}}{\Gamma((q+k)/2)(n-q-k)} \int\limits_{\s^{n-1}\cap H^{\perp}} (\|\cdot\|_K^{-n+q+k})^{\wedge}(\theta) d\theta \, . \nonumber
\end{equation} 
Also for every $m\in \N\cup\{0\}$, $m<(n-k)/2$,
$$ \Delta^m A_{K,H}(0) = \frac{(-1)^m}{(2\pi)^k(n-2m-k)} \int_{\sn\cap H^{\perp}} (\|\cdot\|_K^{-n+2m+k})^{\wedge}(\xi) d\xi \, ,$$
where $\Delta$ denotes the Laplacian on $\R^k$.
\end{pr}

\noindent
We shall also need the following generalization of Brunn's theorem.
\begin{lemma}\textnormal{(\cite{KoldobskyKonigZymonopoulou2008}, Lemma 1)}\label{lemma_brunn}
For a $2$-smooth origin-symmetric convex body $K$ in $\R^n$ the function $A_{K,H}$ is twice differentiable at the origin and 
$$  \Delta A_{K,H}(0) \leq 0 \, .$$
Besides, for any $q\in (0,2)$,
$$ \left\langle \frac{|u|_2^{-q-k}}{\Gamma(-q/2)}, A_{K,H}(u) \right\rangle  \geq 0 \, .$$
\end{lemma}

\begin{theorem}\label{th_kib}
Let $K$ be an origin-symmetric $\kappa$-balanced convex body in $\R^{\kappa n}$. The space $(\R^{\kappa n}, \|\cdot\|_K)$ embeds in $L_{-p}$ for every $p\in [\kappa(n-1)-2, \kappa (n-1)], p>0$. In particular, the body $K$ is a $k$-intersection body for $\kappa(n-1)-2 \leq k \leq \kappa (n-1) , k>0$. 
\end{theorem}

\begin{proof}
It is enough to prove the result for infinitely smooth bodies by Lemma 4.10 in \cite{Koldobsky2005}. Fix $\xi \in \s^{\kappa n -1}$. Since $k$ is positive, this implies $\kappa > 2/(n-1)$. The second equation of Proposition \ref{pr_psf} with $H = H_{\xi}, m=1$, dimension $\kappa n$ instead of $n$, and Lemma \ref{lemma_ftc}, yield 
\begin{align*} 
\Delta A_{K,H_{\xi}}(0) 
	&= \frac{-1}{(2\pi)^{\kappa} (\kappa n- \kappa -2)} \int_{\s^{\kappa n -1}\cap H_{\xi}^{\perp}} (\|\cdot\|_K^{-\kappa n+ \kappa + 2})^{\wedge}(x) dx \\
	&= \frac{-|\s^{\kappa-1}|}{(2\pi)^{\kappa} (\kappa n- \kappa -2)} (\|\cdot\|_K^{-\kappa n+ \kappa + 2})^{\wedge}(\xi)  \, .
\end{align*}
By Lemma \ref{lemma_brunn}, it follows that $(\|\cdot\|_K^{-\kappa n+ \kappa + 2})^{\wedge}(\xi) \geq 0$ for every $\xi \in \s^{\kappa n - 1}$. Hence the distribution $\|\cdot\|_K^{-\kappa (n-1) + 2}$ is positive definite on $\R^{\kappa n}$ and consequently, the body $K$ is a $(\kappa (n-1) - 2)$-intersection body. 

Next, we apply the first equation of Proposition \ref{pr_psf} with $H = H_{\xi}, 0<q<2$, dimension $\kappa n$ instead of $n$, and Lemma \ref{lemma_ftc}, to obtain
$$ \left\langle \frac{|u|_2^{-q-\kappa}}{\Gamma(-q/2)}, A_{K,H_{\xi}}(u) \right\rangle = \frac{2^{-q-\kappa} \pi^{-\kappa/2} |\s^{\kappa-1}|}{\Gamma((q+\kappa)/2)(\kappa n-q-\kappa)}  (\|\cdot\|_K^{-\kappa n+q+\kappa})^{\wedge}(\xi) \, . $$
Using Lemma \ref{lemma_brunn} again, it follows that the space $(\R^{\kappa n}, \|\cdot\|_K)$ embeds in $L_{-\kappa n+q+\kappa}$. With the range of $0<q<2$, this means that every such space embeds in $L_{-p}$ for $p\in (\kappa(n-1)-2, \kappa n-\kappa)$.

Finally, using the second equation of Proposition \ref{pr_psf} as above, but with $m=0$, we obtain 
$$
A_{K,H_{\xi}}(0) = \frac{|\s^{\kappa-1}|}{(2\pi)^{\kappa} (\kappa n-\kappa)} (\|\cdot\|_K^{-\kappa n+ \kappa})^{\wedge}(\xi)  \, ,
$$
which implies that the body $K$ is a $\kappa (n-1)$-intersection body. 
\end{proof}

Let us summarize the above discussion: For an origin-symmetric $\kappa$-balanced convex body $K$ in $\R^{\kappa n}$ the space $(\R^{\kappa n}, \|\cdot\|_K)$ embeds in $L_{-p}$ for every $p>0$, $p\in [\kappa(n-1)-2, \kappa (n-1)]\cup [\kappa n -3, \kappa n)$. For $q>2$, the space $(\R^{\kappa n}, \|\cdot\|_{q, \kappa} )$ does not embed in $L_{-p}$ for $p\in (0, \kappa(n-1)-2)$. For the range $p\in (\kappa n - \kappa, \kappa n -3)$ and $\kappa \geq 4$, the question whether or not for every origin-symmetric $\kappa$-balanced convex body $K$ in $\R^{\kappa n}$ the space $(\R^{\kappa n}, \|\cdot\|_K)$ embeds in $L_{-p}$ remains open.

\begin{co}\label{co_ib}
The only cases where all origin-symmetric convex bodies in $\K^n$ are intersection bodies in $\K^n$ are the following: \\
(i) $n=2, \kappa\in \N$, (ii) $n=3, \kappa\leq 2$ and (iii) $n=4, \kappa =1$. 
\end{co}

\begin{proof}
To obtain the result, we apply Theorem \ref{th_kib} and Proposition \ref{pr_nkib} for different values of $n$. For $n=2$, $\kappa \in [\kappa-2, \kappa]$ for any $\kappa \in \N$. For $n=3$, $\kappa \in [2\kappa-2, 2\kappa]$ only for $\kappa \leq 2$, and for $\kappa>2$, $B_{q,\kappa}$ with $q>2$ is not an intersection body in $\K^3$. For $n=4$, $\kappa \in [3\kappa-2, 3\kappa]$ only for $\kappa = 1$, and for $\kappa>1$, $B_{q,\kappa}$ with $q>2$ is not an intersection body in $\K^4$. For $n\geq 5$, $B^{\kappa n}_q$ with $q>2$ is not an intersection body in $\K^n$ for all $\kappa\in \N$. 
\end{proof}

Another generalization of intersection bodies was introduced by G. Zhang in \cite{Zhang1996} as follows. For $1\leq k \leq n-1$, let $G(n, n-k)$ be the Grassmanian of $(n-k)$-dimensional subspaces of $\R^n$. Recall that the $(n-k)$-dimensional spherical Radon transform, $\RT_{n-k}$, is an operator $\RT_{n-k} \: C(\s^{n-1}) \to C(G(n, n-k))$ defined by
$$ \RT_{n-k} f(H) = \int_{\s^{n-1}\cap H} f(x) dx \, ,$$
for $H\in G(n, n-k)$. Denote the image of the operator $\RT_{n-k}$ by $X$:
$$ \RT_{n-k}(C(\s^{n-1})) = X \subset C(G(n, n-k))\, . $$
Let $M^+(X)$ be the space of positive linear functionals on $X$, that is, for every $\nu \in M^+(X)$ and for every non-negative function $f\in X$, we have $\nu(f)\geq 0$.
\begin{df}
An origin-symmetric star body $K$ in $\R^n$ is called a generalized $k$-intersection body if there exists a functional $\nu\in M^+(X)$ so that for every $f \in C(\s^{n-1})$
$$ \int_{\s^{n-1}} \|x\|_K^{-k} f(x)dx = \nu( \RT_{n-k} f) \, . $$ 
\end{df}
The generalized $k$-intersection bodies are related to the lower-dimensional Busemann-Petty problem, see \cite{Zhang1996}.

P. Goodey and W. Weil proved in \cite{GoodeyWeil1995} that all intersection bodies in $\R^n$ can be obtained as the closure in the radial metric of radial sums of ellipsoids. This result was extended by E. Grinberg and G. Zhang to generalized $k$-intersection bodies with the radial sum replaced by the $k$-radial sum. E. Milman gave a different proof of the latter result in \cite{milman2006}. The complex version of this result was proved in \cite{MR3106735}. We now prove this result in $\K^n$ by adjusting the proof from \cite{KoldobskyYaskin2008} to our setting.

Define the radial sum of two star bodies $K,L$ in $\K^n$, $K +^{\K^n} L$, as a star body in $\K^n$ whose radial function satisfies
$$ \rho^{\kappa}_{K +^{\K^n} L} = \rho^{\kappa}_{K} + \rho^{\kappa}_{L} \, ,$$
or equivalently as
$$ \|\cdot\|^{-\kappa}_{K +^{\K^n} L} = \|\cdot\|^{-\kappa}_{K} + \|\cdot\|^{-\kappa}_{L} \, .$$
We will prove the following theorem in several steps.
\begin{theorem}\label{th_ibaloe}
Let $K$ be an origin-symmetric star body in $\K^n$. Then $K$ is an intersection body in $\K^n$ if and only if $\|\cdot\|^{-\kappa}_{K}$ is the limit, in the space $C_{\kappa}(\s^{\kappa n-1})$, of finite sums of the form
$$ \|\cdot\|^{-\kappa}_{E_1} + \cdots + \|\cdot\|^{-\kappa}_{E_m} \, ,$$
where $E_1, \dots, E_m$ are ellipsoids in $\K^n$.
\end{theorem}

For a vector $\xi$ on the sphere denote by $\{e_1, \dots, e_{\kappa}\}$ an orthonormal basis in $H_{\xi}^{\perp}$. For $a>0, b>0$, let $E_{a,b}(\xi)$ be an ellipsoid in $\R^{\kappa n}$ with the norm 
$$ \|x\|_{E_{a,b}(\xi)}= \left( \frac{\sum_{i=1}^{\kappa} (x, e_i)^2}{a^2} + \frac{|x|_2^2- \sum_{i=1}^{\kappa} (x, e_i)^2}{b^2}\right)^{\frac{1}{2}} \, ,$$
where $x\in\R^{\kappa n}$. Note that $(\sum_{i=1}^{\kappa} (x, e_i)^2)^{1/2}$ is the length of the projection of the vector $x$ onto the subspace $H_{\xi}^{\perp}$. For $\sigma \in \SO(\kappa)$, the projection of $R_{\sigma} x$ onto $H_{\xi}^{\perp}$ has the same length as the projection of $x$ itself, hence $E_{a,b}(\xi)$ is a $\kappa$-balanced ellipsoid or an ellipsoid in $\K^n$.

Recall the formula for the Fourier transform of powers of the Euclidean norm in $\R^n$:
$$ \left(|\cdot|_2^p\right)^{\wedge}(\theta)=\frac{\pi^{\frac{n}{2}} 2^{n+p} \Gamma(\frac{n+p}{2}) }{\Gamma(-\frac{p}{2})} |\theta|_2^{-n-p} \, ,$$
and the formula connecting the Fourier transform and linear transformations
$$ (f(T\cdot))^{\wedge}(y) = |\det T|^{-1} \hat{f}((T^{\ast})^{-1}y) \, ,$$
where $T$ is a linear transformation and $T^{\ast}$ denotes the adjoint of $T$.

\begin{lemma}\label{le_fte}
For $\theta\in \s^{\kappa n-1}$
$$ \left( \|\cdot\|^{-\kappa}_{E_{a,b}(\xi)} \right)^{\wedge}(\theta) = \frac{C(n,\kappa)}{a^{\kappa(n-2)}} \|\theta\|^{-\kappa n +\kappa}_{E_{b, a}(\xi)} \, ,$$
with $C(n,\kappa) = \pi^{\frac{\kappa n}{2}} 2^{\kappa n - \kappa} \Gamma(\frac{\kappa n - \kappa}{2}) / \Gamma(\frac{\kappa}{2})$.
\end{lemma}

\begin{proof}
Let $T$ be a linear operator so that $T B^{\kappa n}_2 = E_{a,b}(\xi)$, then $T$ is a composition of a diagonal operator and a rotation.
\begin{align*}
	\left( \|\cdot\|^{-\kappa}_{E_{a,b}(\xi)} \right)^{\wedge}(\theta) 
	&= \left( |T^{-1} \cdot|^{-\kappa}_2 \right)^{\wedge}(\theta)\\
	&= |\det T^{-1}|^{-1} \left( |\cdot|^{-\kappa}_2 \right)^{\wedge}(T^{\ast}\theta)\\
	&= |\det T|\, C(n, \kappa) |T^{\ast} \theta|^{-\kappa n +\kappa}_2 \\
	&= |\det T|\, C(n, \kappa) \|\theta\|_{ (T^{\ast})^{-1} B^{\kappa n}_2 }^{-\kappa n +\kappa}\\
	&= |\det T|\, C(n, \kappa) \|\theta\|_{E_{\frac{1}{a}, \frac{1}{b}}(\xi)}^{-\kappa n +\kappa} \\
	&= |\det T|\, C(n, \kappa) \|ab \theta\|_{E_{b, a}(\xi)}^{-\kappa n +\kappa} \\
	&= \frac{C(n, \kappa)}{a^{\kappa n-2\kappa}} \|\theta\|_{E_{b, a}(\xi)}^{-\kappa n +\kappa} \, .
\end{align*}
\end{proof}

\begin{lemma}\label{le_sbabi}
Let $K$ be an origin-symmetric star body in $\K^n$, then $\|\cdot\|^{-\kappa}_K$ can be approximated in the space of $C_{\kappa}(\s^{\kappa n -1})$ by functions of the form
$$ f_{a,b}(\xi) = \frac{1}{a^{\kappa (n-2)}} \int_{\s^{\kappa n -1}} \|\theta \|_K^{-\kappa} \|\theta\|^{-\kappa n +\kappa}_{E_{b, a}(\xi)} d\theta $$
for an appropriate choice of $b$ and $a\rightarrow 0$.
\end{lemma}

\begin{proof}
Using the formula for the Fourier transform of powers of the Euclidean norm, Parseval's formula on the sphere and previous lemma, we get
\begin{align*}
	\int_{\s^{\kappa n -1}} \|\theta\|^{-\kappa n +\kappa}_{E_{b, a}(\xi)} d\theta 
		&= \frac{C(n,\kappa)}{(2\pi)^{\kappa n}} \int_{\s^{\kappa n -1}} \|\theta\|^{-\kappa n +\kappa}_{E_{b, a}(\xi)} (|\cdot|_2^{-\kappa n +\kappa})^{\wedge}(\theta) d\theta \\
		&= \frac{C(n,\kappa)}{(2\pi)^{\kappa n}} \int_{\s^{\kappa n -1}} (\|\cdot\|^{-\kappa n +\kappa}_{E_{b, a}(\xi)})^{\wedge}(\theta) d\theta \\
  	&= a^{\kappa (n-2)} \int_{\s^{\kappa n-1}} \| \theta \|^{-\kappa}_{E_{a, b}(\xi)} d\theta \, . 
\end{align*}
Thus 
$$ \frac{1}{a^{\kappa (n-2)}} \!\! \int\limits_{\s^{\kappa n -1}} \|\theta\|^{-\kappa n +\kappa}_{E_{b, a}(\xi)} d\theta = \!\! \int\limits_{\s^{\kappa n-1}} \left( \frac{ \sum_{i=1}^{\kappa} (\theta, e_i)^2 }{a^2} + \frac{1- \sum_{i=1}^{\kappa} (\theta, e_i)^2}{b^2} \right)^{\frac{-\kappa}{2}} \!\! d\theta \, .$$

Note that for a fixed $a$ this integral approaches infinity as $b\rightarrow \infty$ and it goes to zero as $b\rightarrow 0$. Hence for every $a$ there exists $b=b(a)$ such that
$$  \frac{1}{a^{\kappa (n-2)}} \int_{\s^{\kappa n -1}} \|\theta\|^{-\kappa n +\kappa}_{E_{b(a), a}(\xi)} d\theta = 1 \, .$$
Since the measure in the above integral is rotation invariant, $b(a)$ does not depend on $\xi$. Hence for every $\xi$ on the sphere and for any $\delta \in (0,1)$, we have
\begin{align*}
	& \left| \|\xi \|_K^{-\kappa} - \frac{1}{a^{\kappa (n-2)}} \int_{\s^{\kappa n -1}} \|\theta \|_K^{-\kappa} \|\theta\|^{-\kappa n +\kappa}_{E_{b(a), a}(\xi)} d\theta\right| \\
	&\leq \frac{1}{a^{\kappa (n-2)}} \int_{\s^{\kappa n -1}} \left| \|\xi \|_K^{-\kappa} - \|\theta \|_K^{-\kappa} \right| \|\theta\|^{-\kappa n +\kappa}_{E_{b(a), a}(\xi)} d\theta \\
	&= \frac{1}{a^{\kappa (n-2)}} \left(\int_{\sum\limits_{i=1}^{\kappa} (\theta, e_i)^2\geq \delta} + \int_{\sum\limits_{i=1}^{\kappa} (\theta, e_i)^2 < \delta}\right) \left| \|\xi \|_K^{-\kappa} - \|\theta \|_K^{-\kappa} \right| \|\theta\|^{-\kappa n +\kappa}_{E_{b(a), a}(\xi)} d\theta \\
	&= I_1+I_2 \, .
\end{align*}

By the uniform continuity of the function $\|\cdot \|_K$ on the sphere, for any $\epsilon>0$, there is $\delta \in (0,1)$, $\delta$ close to one, so that $|\|\xi\|^{-\kappa}_K - \|\theta\|^{-\kappa}_K|<\frac{\epsilon}{2}$ for $(\theta, \xi)\geq \delta^{\frac{1}{2}}$. For $\theta$ on the sphere with $\sum_{i=1}^{\kappa} (\theta, e_i)^2\geq \delta$, let $\sigma_1 \in \SO(\kappa)$ be so that $(R_{\sigma_1}\theta, e_i)=0$ for $i\neq 1$. Since the length of the projection onto $ H_{\xi}^{\perp} $ of $\theta$ and $ R_{\sigma_1}\theta $ is the same, it follows that $(R_{\sigma_1}\theta, e_1)\geq \delta^{\frac{1}{2}}$. To see that such rotation $\sigma_1$ exists, write $\theta=\theta_{\xi} + \theta_{\xi^{\perp}}$, with $\theta_{\xi}\in H_{\xi}, \theta_{\xi^{\perp}}\in H_{\xi}^{\perp}$. Observe that for any two vectors $u,v \in H_{\xi}^{\perp}$ with $\|u\|=\|v\|$ there exists $\sigma$ so that $R_{\sigma}u=v$. In particular, there exists $\sigma_1$ so that $R_{\sigma_1} \frac{\theta_{\xi^{\perp}}}{\|\theta_{\xi^{\perp}}\|} = e_1$. Since $R_{\sigma} \theta_{\xi}\in H_{\xi}, R_{\sigma} \theta_{\xi^{\perp}}\in H_{\xi}^{\perp}$ for any $\sigma$, we have $ (R_{\sigma_1}\theta, e_i)=(R_{\sigma_1}\theta_{\xi} + R_{\sigma_1}\theta_{\xi^{\perp}}, e_i)=(R_{\sigma_1}\theta_{\xi^{\perp}}, e_i) $. Let $\sigma_2$ be so that $R_{\sigma_2} e_1 = \xi$. Such $\sigma_2$ exists since $\xi, e_1 \in H_{\xi}^{\perp}$. Then 
$$ (R_{\sigma_1}\theta, e_1) = (R_{\sigma_2} R_{\sigma_1}\theta, R_{\sigma_2} e_1) = (R_{\sigma_2} R_{\sigma_1}\theta, \xi) \geq \delta^{\frac{1}{2}} .$$
Since $K$ is $\kappa$-balanced, we obtain 
$$|\|\xi\|^{-\kappa}_K - \|\theta\|^{-\kappa}_K| = |\|\xi\|^{-\kappa}_K - \|R_{\sigma_2} R_{\sigma_1}\theta\|^{-\kappa}_K|<\frac{\epsilon}{2} \, .$$ 
Thus with this choice of $\delta$, we can estimate the first integral as follows:
\begin{align*}
	I_1 &= \frac{1}{a^{\kappa (n-2)}} \int_{\sum_{i=1}^{\kappa} (\theta, e_i)^2\geq \delta} \left| \|\xi \|_K^{-\kappa} - \|\theta \|_K^{-\kappa} \right| \|\theta\|^{-\kappa n +\kappa}_{E_{b(a), a}(\xi)} d\theta \\
	&< \frac{\epsilon}{2} \frac{1}{a^{\kappa (n-2)}} \int_{\sum_{i=1}^{\kappa} (\theta, e_i)^2 \geq \delta} \|\theta\|^{-\kappa n +\kappa}_{E_{b(a), a}(\xi)} d\theta \leq \frac{\epsilon}{2} \, .
\end{align*}
Next, estimate the second integral as follows:
\begin{align*}
	I_2 &= \frac{1}{a^{\kappa (n-2)}} \int_{\sum_{i=1}^{\kappa} (\theta, e_i)^2 < \delta} \left| \|\xi \|_K^{-\kappa} - \|\theta \|_K^{-\kappa} \right| \|\theta\|^{-\kappa n +\kappa}_{E_{b(a), a}(\xi)} d\theta \\
	&\leq \frac{2 \max\limits_{x\in\s^{\kappa n-1}}  \|x\|_K^{-\kappa}}{a^{\kappa (n-2)}} \int_{\sum_{i=1}^{\kappa} (\theta, e_i)^2 < \delta} \|\theta\|^{-\kappa n +\kappa}_{E_{b(a), a}(\xi)} d\theta \\
	&\leq \frac{2 \max\limits_{x\in\s^{\kappa n-1}}  \|x\|_K^{-\kappa}}{a^{\kappa (n-2)}} \int_{\sum_{i=1}^{\kappa} (\theta, e_i)^2 < \delta} \left(\frac{1- \sum_{i=1}^{\kappa} (\theta, e_i)^2}{a^2} \right)^{\frac{-\kappa n +\kappa}{2}} d\theta \\
	&= 2 a^{\kappa} \max\limits_{x\in\s^{\kappa n-1}}  \|x\|_K^{-\kappa} \int_{\sum_{i=1}^{\kappa} (\theta, e_i)^2 < \delta} \left(1- \sum_{i=1}^{\kappa} (\theta, e_i)^2 \right)^{\frac{-\kappa n +\kappa}{2}} d\theta \\
	&\leq 2 a^{\kappa} \max\limits_{x\in\s^{\kappa n-1}}  \|x\|_K^{-\kappa} |\s^{\kappa n-1}| \left(1-\delta \right)^{\frac{-\kappa n +\kappa}{2}} \, .
\end{align*}
Now we can choose $a$ so small that $I_2 \leq \frac{\epsilon}{2}$.
\end{proof}

\begin{lemma}\label{le_abem}
Let $\mu$ be a finite measure on $\s^{\kappa n -1}$ and $a, b>0$. The function 
$$ f(\xi) = \int_{\s^{\kappa n -1}} \|\theta \|^{-\kappa}_{E_{a,b}(\xi)} d\mu(\theta) $$
is the limit, in the space $C_{\kappa}(\s^{\kappa n -1})$, of sums of the form
$$ \sum\limits_{i=1}^{m} \|\xi\|^{-\kappa}_{E_i} \, ,$$
where $E_1, \dots, E_m$ are $\kappa$-balanced ellipsoids. 
\end{lemma}

\begin{proof}
For $\epsilon > 0 $, choose a finite covering of the sphere by spherical $\epsilon$-balls: $B_{\epsilon}(\xi_i) = \{\theta\in\s^{\kappa n -1} \: |\theta-\xi_i|_2 < \epsilon \}$, $\xi_i\in\s^{\kappa n -1}$, $i=1, \dots, m = m(\epsilon)$. Define
$$ \tilde{B}_{\epsilon}(\xi_1)= B_{\epsilon}(\xi_1) \hspace{3mm} \text{ and }  \hspace{3mm} \tilde{B}_{\epsilon}(\xi_i)= B_{\epsilon}(\xi_i) \setminus \bigcup\limits_{j=1}^{i-1} B_{\epsilon}(\xi_j)\, ,  \hspace{3mm} \text{for i=2, \dots, m}\, .$$ 
Set $p_i=\mu(\tilde{B}_{\epsilon}(\xi_i))$, then $p_1 + \cdots + p_m = \mu(\s^{\kappa n -1})$. 

Denote by $\rho(E_{a,b}(\xi), x)$ the value of the radial function of the ellipsoid $E_{a,b}(\xi)$ at the point $x\in \s^{\kappa n -1}$. Observe that the length of the projection of the vector $x$ onto $H_{\xi}^{\perp}$ equals to  
the length of the projection of the vector $\xi$ onto $H_{x}^{\perp}$. Since $\rho(E_{a,b}(\xi), x)$ depends only on the length of the projection of $x$ onto $H_{\xi}^{\perp}$, it follows that $\rho(E_{a,b}(\xi), x) = \rho(E_{a,b}(x), \xi)$. Hence
$$ |\rho^{\kappa}(E_{a,b}(\xi), x) - \rho^{\kappa}(E_{a,b}(\theta), x)| \leq C_{a,b} |\xi - \theta| \, ,$$
for some constant $C_{a,b}$ depending only on $a$ and $b$. We are now ready to estimate
\begin{align*}
 & \left| \int_{\s^{\kappa n - 1}} \rho^{\kappa}(E_{a,b}(\xi), x) d\mu(\xi) - \sum\limits_{i=1}^{m} p_i \,  \rho^{\kappa}(E_{a,b}(\xi_i), x) \right| \\
 &= \left| \sum\limits_{i=1}^{m} \left( \int_{\tilde{B}_{\epsilon}(\xi_i)} \rho^{\kappa}(E_{a,b}(\xi), x) d\mu(\xi) -  \int_{\tilde{B}_{\epsilon}(\xi_i)} \rho^{\kappa}(E_{a,b}(\xi_i), x) d\mu(\xi) \right) \right| \\
 &\leq \sum\limits_{i=1}^{m} \int_{\tilde{B}_{\epsilon}(\xi_i)} |\rho^{\kappa}(E_{a,b}(\xi), x) -  \rho^{\kappa}(E_{a,b}(\xi_i), x) | d\mu(\xi) \\
 &\leq \sum\limits_{i=1}^{m} \int_{\tilde{B}_{\epsilon}(\xi_i)} C_{a,b} |\xi - \xi_i| d\mu(\xi) \\
 &\leq \epsilon \, C_{a,b} \, \mu(\s^{\kappa n - 1})\,.
\end{align*}
The result follows by letting $\epsilon \rightarrow 0$ and defining $ \|\xi\|^{-\kappa}_{E_i} = p_i \, \rho^{\kappa}(E_{a,b}(\xi_i), x)$.
\end{proof}

\noindent
\begin{proof}[Proof of Theorem \ref{th_ibaloe}]
The 'if' part follows immediately from Theorem \ref{th_kibpdd} and Lemma \ref{le_fte}.

To prove the converse, suppose $K$ is an intersection body in $\K^n$. By Lemma \ref{le_sbabi}, $\|\xi\|^{-\kappa}_K$ is a uniform limit of functions of the form
$$ \frac{1}{a^{\kappa (n-2)}} \int_{\s^{\kappa n -1}} \|\theta \|_K^{-\kappa} \|\theta\|^{-\kappa n +\kappa}_{E_{b, a}(\xi)} d\theta  \, ,$$ 
as $a \rightarrow 0$. By Parseval's formula on the sphere, this equals to
$$ \frac{1}{C(n,\kappa)} \int_{\s^{\kappa n -1}} \|\theta\|^{-\kappa}_{E_{a, b}(\xi)} (\|\cdot \|_K^{-\kappa})^{\wedge}(\theta) d\theta = \int_{\s^{\kappa n -1}} \|\theta\|^{-\kappa}_{E_{a, b}(\xi)}  d\mu(\theta) \, ,$$
where $d\mu = 1/C(n,\kappa) (\|\cdot \|_K^{-\kappa})^{\wedge}$. By Lemma \ref{le_abem}, the above is the uniform limit of sums of the form
$$ \|\xi\|^{-\kappa}_{E_1} + \cdots + \|\xi\|^{-\kappa}_{E_m}\, .$$
\end{proof}

\begin{pr}\label{pr_ibgkib}
An origin-symmetric star body in $\K^n$ is an intersection body in $\K^n$ if and only if it is a $\kappa$-balanced generalized $\kappa$-intersection body in  $\R^{\kappa n}$. 
\end{pr}

\begin{proof}
By Theorem \ref{th_ibaloe}, $M$ is an intersection body in $\K^n$ if and only if $\|\cdot\|^{-\kappa}_{M}$ is the limit, in the space $C_{\kappa}(\s^{\kappa n-1})$, of $\kappa$-radial sums of ellipsoids in $\K^n$. The latter condition is true if and only if $M$ is a $\kappa$-balanced generalized $\kappa$-intersection body in  $\R^{\kappa n}$, by a result of E. Grinberg and G. Zhang, see Theorem 6.4 in \cite{MR1658156}.
\end{proof}

Together Proposition \ref{pr_ibgkib} and Corollary \ref{co_ibkib} imply that for $\kappa$-balanced origin-symmetric star bodies in $\R^{\kappa n}$ the class of $\kappa$-intersection bodies and the class of generalized $\kappa$-intersection bodies coincide. This is not true in general as was shown by E. Milman in \cite{milman2008}. 

\section{Stability in the Busemann-Petty problem and hyperplane  inequalities}\label{section_StabilityInTheBusemannPettyProblemAndHyperplaneInequalities}

Intersection bodies played an important role in the solution of the Busemann-Petty problem, which can be formulated as follows. Given two origin-sym\-met\-ric convex bodies $K$ and $L$ in $\Rn$ such that for every $\xi \in \s^{n-1}$
$$ |K\cap \xi^{\perp}| \leq |L \cap \xi^{\perp}| \, , $$
does it follow that 
$$ |K| \leq |L|? $$
The answer is affirmative for $n\leq 4$ and negative for $n\geq 5$. This problem, posed in 1956 in \cite{BusemannPetty1956}, was solved in the late 90's as a result of a sequence of papers 
\cite{LR,Ba,Lu,Gi,Bo,Pa,Ga1,Ga2,Zh1,K1,K2,GardnerKoldobskySchlumprecht1999,Zh2}, see \cite{Koldobsky2005}, p. 3-5, for the history of the solution. One of the main steps in the solution was the connection established by E. Lutwak in \cite{Lu} between this problem and intersection bodies: For an intersection body $K$ and any star body $L$ the Busemann-Petty problem has a positive answer. For any origin-symmetric convex body $L$ that is not an intersection body, one can construct a body $K$, giving together with $L$ a counterexample. The complex version of this problem was solved in \cite{KoldobskyKonigZymonopoulou2008}.

The Busemann-Petty problem in $\K^n$ can be formulated as follows: Given two origin-symmetric $\kappa$-balanced convex bodies $K$ and $L$ in $\R^{\kappa n}$ such that $ |K\cap H_{\xi}| \leq |L \cap H_{\xi}| $, for every $\xi \in \s^{\kappa n-1}$. Does it follow that $|K| \leq |L|$? We prove below that the answer is affirmative in the following cases: (i) $n=2$, $\kappa \in \N$, (ii) $n=3$, $\kappa \leq 2$, (iii) $n=4$, $\kappa=1$, and negative for any other values of $n$ and $\kappa$, see also \cite{rubin2010}. The solution uses a connection with intersection bodies in $\K^n$, analogous to Lutwak's connection: If $K$ is an 
intersection body in $\K^n$ and $L$ is any star body in $\K^n$, then the Busemann-Petty problem in $\K^n$ has an affirmative answer. If there exists an origin-symmetric convex body $L$ in $\K^n$ that is not an intersection body in $\K^n$, then one can construct another origin-symmetric convex body $K$ in $\K^n$, so that the pair of bodies $K, L$ provides a counterexample.

A. Zvavitch generalized the Busemann-Petty problem in $\R^n$ to arbitrary measures in place of volume and proved that the answer is affirmative for $n\leq 4$ and negative for $n\geq 5$, see \cite{Zvavitch2005}. The complex version of this result was proved in \cite{Zymonopoulou2008}. In this section we extend Zvavitch's result to $\K^n$ and consider the associated stability question as well as the stability in the Busemann-Petty problem in $\K^n$. Stability in the original Busemann-Petty problem was established in \cite{MR2836117}, for the complex version in \cite{koldobsky2011} and for arbitrary measures in \cite{MR2891246}, other stability results include \cite{MR3022986, MR3009547}.

We start with the stability consideration in the Busemann-Petty problem in $\K^n$.

\begin{pr}\label{pr_sfbpikn}
Let $K, L$ be origin-symmetric star bodies in $\K^n$ and let $\epsilon>0$. Suppose $K$ is an intersection body in $\K^n$ and for every $\xi \in \s^{\kappa n -1}$
$$ |K\cap H_{\xi}| \leq |L \cap H_{\xi}| + \epsilon \, , $$
then
$$ |K|^{\frac{n-1}{n}} \leq |L|^{\frac{n-1}{n}} + \epsilon \, \frac{|B^{\kappa n}_2|^{\frac{n-1}{n}}}{|B^{\kappa n-\kappa}_2|} \, .$$
\end{pr}

\begin{proof}
By (\ref{eq_var}) the inequality for sections can be written as
$$ \RT^{\kappa}(\|\cdot\|_K^{-\kappa n + \kappa})(\xi) \leq  \RT^{\kappa}(\|\cdot\|_L^{-\kappa n + \kappa})^{\wedge}(\xi) + \epsilon \, \kappa (n - 1)\, .$$ 
Let $\nu$ be the measure which corresponds to the body $K$ by Definition \ref{df_ib}. Integrating the above inequality over the sphere with respect to $\nu$ and applying the equality condition of Definition \ref{df_ib}, yields
$$ \int\limits_{\s^{\kappa n-1}} \|x\|^{-\kappa n}_K dx \leq \int\limits_{\s^{\kappa n-1}} \|x\|^{-\kappa}_K \|x\|_L^{-\kappa n + \kappa} dx + \epsilon \kappa (n - 1) \int\limits_{\s^{\kappa n-1}} d\nu(\xi) \, .$$
Applying H{\"o}lder's inequality and using polar formula for the volume, gives
$$ \kappa n |K| \leq \kappa n |K|^{\frac{1}{n}} |L|^{\frac{n-1}{n}} + \epsilon \kappa (n - 1) \int\limits_{\s^{\kappa n-1}} d\nu(\xi) \, . $$
The spherical Radon transform on $\K^n$ of the constant function one, is the constant function with value $|\s^{\kappa n -\kappa-1}|$. Using the equality condition of Definition \ref{df_ib} and H{\"o}lder's inequality, we obtain
\begin{align*}
	\int\limits_{\s^{\kappa n -1}} d\nu(\xi)
	&= \frac{1}{|\s^{\kappa n -\kappa-1}|} \int\limits_{\s^{\kappa n -1}} \RT^{\kappa} 1(\xi) d\nu(\xi) \\
	&= \frac{1}{|\s^{\kappa n -\kappa-1}|} \int\limits_{\s^{\kappa n -1}} \|x\|_K^{-\kappa} dx \\
	&\leq \frac{1}{|\s^{\kappa n -\kappa-1}|} \kappa n |K|^{\frac{1}{n}} |B_2^{\kappa n}|^{\frac{n-1}{n}}  \, .
\end{align*}
Altogether, we have
$$ |K|^{\frac{n-1}{n}} \leq |L|^{\frac{n-1}{n}} + \epsilon \, \frac{|B^{\kappa n}_2|^{\frac{n-1}{n}}}{|B^{\kappa n-\kappa}_2|} \, . $$
\end{proof}

\begin{pr}\label{pr_nbp}
Suppose that $L$ be an origin-symmetric convex body in $\K^n$ that is not an intersection body in $\K^n$. Then there is an origin-symmetric convex body $K$ in $\K^n$ satisfying 
$$ |K\cap H_{\xi}| \leq |L \cap H_{\xi}| \, , $$
for every $\xi \in \s^{\kappa n -1}$, but
$$ |K| > |L|  \, .$$
\end{pr}

\begin{proof}
By Lemma 4.10 in \cite{Koldobsky2005} it is enough to prove the result for infinitely smooth bodies with strictly positive curvature. From our assumptions about the body $L$, it follows that the Fourier transform of $\|\cdot\|^{-\kappa}_L$ is an infinitely-smooth $\kappa$-invariant function on $\R^n\setminus \{0\}$, homogeneous of degree $-\kappa n + \kappa$. Moreover, this function is negative on some open $\kappa$-balanced subset $\Omega$ of the sphere $\s^{\kappa n-1}$. Choose a smooth non-negative $\kappa$-invariant function $f$ on $\s^{\kappa n-1}$ whose  support is contained in $\Omega$ and extend $f$ to a $\kappa$-invariant homogeneous function $f\left(\frac{x}{|x|}\right) |x|^{-\kappa}$ of degree $-\kappa$ on $\R^{\kappa n}$. The Fourier transform of this extension is an infinitely smooth $\kappa$-invariant function on $\R^{\kappa n}\setminus\{0\}$,  homogeneous of degree $-\kappa n+\kappa$:
$$ \left( f\left(\frac{x}{|x|}\right) |x|^{-\kappa} \right)^{\wedge} (y) = g\left(\frac{y}{|y|}\right) |y|^{-\kappa n+\kappa}  ,$$
with $g\in C^{\infty}(\s^{\kappa n-1})$. Since $f$ is $\kappa$-invariant, so is $g$. Define an origin-symmetric body $K$ in $K^n$ by
$$ \|x\|_K^{-\kappa n+\kappa}  = \|x\|_L^{-\kappa n+\kappa} -\epsilon g\left(\frac{x}{|x|}\right) |x|^{-\kappa n+\kappa} \, ,$$
for some $\epsilon >0$, small enough to ensure that the body $K$ is convex. Taking the Fourier transform of both sides in the above equation, yields 
\begin{align}
(\|\cdot\|_K^{-\kappa n+\kappa})^{\wedge}(\xi)  
	&= (\|\cdot\|_L^{-\kappa n+\kappa})^{\wedge}(\xi) -\epsilon (2\pi)^{\kappa n} f(\xi) \label{eq_ifft} \\ 
	&\leq (\|\cdot\|_L^{-\kappa n+\kappa})^{\wedge}(\xi) \, , \nonumber
\end{align}
for $\xi \in \s^{\kappa n -1}$. It follows by Theorem \ref{thm_vos}, that for $\xi \in \s^{\kappa n -1}$
$$ |K\cap H_{\xi}| \leq |L \cap H_{\xi}| \, .$$
Next, multiply (\ref{eq_ifft}) by $(\|\cdot\|^{-\kappa}_L)^{\wedge}$ and integrate over the sphere $\s^{\kappa n -1}$. We obtain
\begin{align*}
\int_{\s^{\kappa n -1}} (\|\cdot\|_K^{-\kappa n+\kappa})^{\wedge}(\xi) (\|\cdot\|^{-\kappa}_L)^{\wedge}(\xi) d\xi 
	&= \int_{\s^{\kappa n -1}} (\|\cdot\|_L^{-\kappa n+\kappa})^{\wedge}(\xi) (\|\cdot\|^{-\kappa}_L)^{\wedge}(\xi) d\xi \\
	&- \epsilon (2\pi)^{\kappa n} \int_{\s^{\kappa n -1}} f(\xi) (\|\cdot\|^{-\kappa}_L)^{\wedge}(\xi) d\xi \\
	&>  \int_{\s^{\kappa n -1}} (\|\cdot\|_L^{-\kappa n+\kappa})^{\wedge}(\xi) (\|\cdot\|^{-\kappa}_L)^{\wedge}(\xi) d\xi \, .
\end{align*}
since $(\|\cdot\|^{-\kappa}_L)^{\wedge} <0$ on the support of $f$. This inequality, by means of Parseval's formula on the sphere, see Lemma \ref{lemma_ParsevalOnTheSphere}, H{\"o}lder's inequality and the polar formula for the volume, results in 
$$ |K| > |L| \, ,$$
proving the claim.
\end{proof}

Propositions \ref{pr_sfbpikn} and \ref{pr_nbp} imply that the Busemann-Petty problem in $\K^n$ has a positive answer if and only if every origin-symmetric convex body in $\K^n$ is an intersection body in $\K^n$. Hence by means of Corollary \ref{co_ib}, we obtain
\begin{theorem}\label{th_bp}
The Busemann-Petty problem in $\K^n$ has an affirmative answer only in the following cases: (i) $n=2$, $\kappa \in \N$, (ii) $n=3$, $\kappa \leq 2$, (iii) $n=4$, $\kappa=1$, and a negative answer for any other values of $n$ and $\kappa$. 
\end{theorem}

\begin{re}
It is worth mentioning that as special cases, Theorem \ref{th_bp} implies an affirmative answer to the lower-dimensional Busemann-Petty problems in $\R^6$ for sections of dimension $3$ and in $\C^4$ for sections of complex dimension $2$ for convex bodies with additional symmetries. Recall that the lower-dimensional Busemann-Petty problem in $\R^n$ with $n\geq 5$ is open for sections of dimension $k=2,3$ and it is open in $\C^n$ with $n\geq 4$ for sections of complex dimension $k=2$.\footnote{The proof in \cite{MR3149693} can be easily adjusted to show the negative answer to the lower-dimensional Busemann-Petty problem in $\C^n$ for sections of complex dimension $3\leq k \leq n-2$.} 
\end{re}

The question of stability in Busemann-Petty problems leads to hyperplane inequalities. These are related to the famous Hyperplane Conjecture. This conjecture can be formulated as follows. Does there exist an absolute constant $C$ so that for any origin-symmetric convex body $K$ in $\R^n$
$$ |K|^{\frac{n-1}{n}} \leq C \max\limits_{\xi\in\s^{n-1}} |K\cap \xi^{\perp}| \, ?$$
Here $\xi^{\perp}$ stands for the central hyperplane perpendicular to $\xi$. This problem remains open. The best known estimate $C \sim n^{1/4}$ is due to B. Klartag \cite{klartag2006}, who slightly improved the previous estimate of J. Bourgain \cite{bourgain1991}. 

Interchanging the roles of $K$ and $L$ in Proposition \ref{pr_sfbpikn} and letting 
$$\epsilon = \max_{\xi\in\s^{\kappa n -1}} \left| |K\cap H_{\xi}| - |L \cap H_{\xi}| \right| \, ,$$ 
we obtain the corresponding volume difference inequality.
\begin{co}
Suppose $K, L$ are intersection bodies in $\K^n$, then
$$ \left| |K|^{\frac{n-1}{n}} - |L|^{\frac{n-1}{n}} \right| \leq  \frac{|B^{\kappa n}_2|^{\frac{n-1}{n}}}{|B^{\kappa n-\kappa}_2|} \max_{\xi\in\s^{\kappa n -1}} \left| |K\cap H_{\xi}| - |L \cap H_{\xi}| \right| \, .$$
\end{co}

\noindent
Setting $L=\delta B^{\kappa n}_2$ and letting $\delta$ go to zero, we obtain:
\begin{co}\label{co_hibp}
Suppose $K$ is an intersection body in $\K^n$, then
$$ |K|^{\frac{n-1}{n}} \leq  \frac{|B^{\kappa n}_2|^{\frac{n-1}{n}}}{|B^{\kappa n-\kappa}_2|} \max_{\xi\in\s^{\kappa n -1}} |K\cap H_{\xi}| \, .$$
\end{co} 
Recall that $ e^{-\frac{\kappa}{2}} < \frac{|B^{\kappa n}_2|^{\frac{n-1}{n}}}{|B^{\kappa n-\kappa}_2|} < 1$. The upper bound follows easily from the log-convexity of the Gamma function. For the lower bound see Lemma 2.1 in \cite{MR1796717}.

For $\kappa=1,2$, Corollary \ref{co_hibp} reduces to the previously known hyperplane inequalities corresponding to the stability problem in the original \cite{MR2836117} and in the complex version \cite{koldobsky2011} of Busemann-Petty problem.

Now we turn to the Busemann-Petty problem in $\K^n$ for arbitrary measures. 
Let $f, g$ be non-negative even continuous functions on $\R^{\kappa n} \setminus \{0\}$, that are locally-integrable on every line through the origin. Let $\mu$ be an absolutely continuous measure with respect to the Lebesgue measure on $\R^{\kappa n}$ with the density $f$. Define a measure $\gamma$ on $H_{\xi}$, for any $\xi\in\s^{\kappa n -1}$, by
$$ \gamma(B) = \int_B g(x) dx \, ,$$
for any bounded Borel set $B\subset H_{\xi}$. The Busemann-Petty problem in $\K^n$ for arbitrary measures can be formulated as follows:

Given $K, L$ two origin-symmetric $\kappa$-balanced convex bodies in $\R^{\kappa n}$ satisfying 
$$ \gamma(K\cap H_{\xi}) \leq \gamma(L\cap H_{\xi}) \, ,$$
for every $\xi\in\s^{\kappa n -1}$, does it follow that
$$ \mu(K) \leq \mu(L) \, ?$$

Since we work with $\kappa$-balanced sets, we can assume that the measures $\mu, \gamma$ are $\kappa$-invariant, consequently the functions $f, g$ are $\kappa$-invariant as well. We need a polar formula for the measure of star bodies in $\K^n$ as well as for the measure of their sections.
\begin{equation}\label{eq_msrebody}
	\mu(K) = \int_K f(x) dx = \int_{\s^{\kappa n -1}} \int_0^{\|x\|_K^{-1}} f(rx) r^{\kappa n -1} dr dx \, . 
\end{equation}
By Lemma \ref{lemma_rfts}, for any $\xi\in\s^{\kappa n -1}$, we have
\begin{align}
	\gamma(K\cap H_{\xi}) 
	&= \int_{\s^{\kappa n -1}\cap H_{\xi}} \int_0^{\|x\|_K^{-1}} g(rx) r^{\kappa n-\kappa -1} dr dx \nonumber \\ 
	&= \int_{\s^{\kappa n -1}\cap H_{\xi}} \left(|x|_2^{- \kappa n + \kappa} \int_0^{\frac{|x|_2}{\|x\|_K}} g\left(\frac{rx}{|x|_2}\right) r^{\kappa n-\kappa -1} dr \right) dx \nonumber \\ 
	&= \RT^{\kappa} \left(|x|_2^{- \kappa n + \kappa} \int_0^{\frac{|x|_2}{\|x\|_K}} g\left(\frac{rx}{|x|_2}\right) r^{\kappa n-\kappa -1} dr \right) (\xi) \nonumber \\ 
	&=\frac{|S^{\kappa-1}|}{(2 \pi)^{\kappa}} \left(|x|_2^{- \kappa n + \kappa} \int_0^{\frac{|x|_2}{\|x\|_K}} g\left(\frac{rx}{|x|_2}\right) r^{\kappa n-\kappa -1} dr \right)^{\wedge}(\xi) \, . \label{eq_msresection}
\end{align}

The following elementary lemma is an analog of a lemma used by A. Zvavitch in \cite{Zvavitch2005}.
\begin{lemma}\label{lemma_elem}
Let $\kappa, n \in \N$, $\kappa \geq 1, n \geq 2$ and let $a,b \geq 0$. For non-negative integrable functions $\alpha, \beta$ on $[0, \max\{a,b\}]$ so that $t^{\kappa} \frac{\alpha(t)}{\beta(t)}$ is non-decreasing, we have
$$ a^{\kappa} \frac{\alpha(a)}{\beta(a)} \int_a^b t^{\kappa n - \kappa -1} \beta(t) dt \leq \int_a^b t^{\kappa n -1} \alpha(t) dt  \, .$$
\end{lemma}

\begin{proof} Compute
\begin{align*}
 a^{\kappa} \frac{\alpha(a)}{\beta(a)} \int_a^b t^{\kappa n - \kappa -1} \beta(t) dt
	&= \int_a^b t^{\kappa n -1} \alpha(t) \left(a^{\kappa} \frac{\alpha(a)}{\beta(a)} \right) \left(t^{\kappa} \frac{\alpha(t)}{\beta(t)} \right)^{-1} dt \\
 	&\leq \int_a^b t^{\kappa n -1} \alpha(t) dt  \, .
\end{align*}
\end{proof}

\begin{pr}\label{pr_positive}
Let $\epsilon>0$ and let $f,g$ be even non-negative $\kappa$-invariant continuous functions on $\R^{\kappa n}\setminus \{0\}$ and so that for any fixed $x\in\s^{\kappa n -1}$, $f(tx), g(tx)$ are locally-integrable in $t$ and $t^{\kappa} \frac{f(tx)}{g(tx)}$ is a non-decreasing function in $t$. Suppose that an origin-symmetric star body $K$ in $\K^n$ has the property that $\|x\|_K^{-\kappa} \frac{f(x\|x\|_K^{-1})}{g(x\|x\|_K^{-1})}$ is a positive-definite distribution on $\R^{\kappa n}$. Then for any origin-symmetric star body $L$ in $\K^n$ satisfying
$$ \gamma(K\cap H_{\xi}) \leq \gamma(L\cap H_{\xi}) + \epsilon  \, ,$$
for every $\xi\in\s^{\kappa n -1}$, it follows that
$$ \mu(K) \leq \mu(L) + \epsilon \frac{1}{|\s^{\kappa n -\kappa-1}|} \int\limits_{\s^{\kappa n -1}} \|x\|_K^{-\kappa} \frac{f(x\|x\|_K^{-1})}{g(x\|x\|_K^{-1})} dx \, .$$ 
\end{pr}

\begin{proof}
Using equation (\ref{eq_msresection}), the inequality for sections can be written as
\begin{align*} 
	& \RT^{\kappa} \left(|x|_2^{- \kappa n + \kappa} \int_0^{\frac{|x|_2}{\|x\|_K}} g\left(\frac{rx}{|x|_2}\right) r^{\kappa n-\kappa -1} dr \right)(\xi)\\ 
	&\leq \RT^{\kappa} \left(|x|_2^{- \kappa n + \kappa} \int_0^{\frac{|x|_2}{\|x\|_L}} g\left(\frac{rx}{|x|_2}\right) r^{\kappa n-\kappa -1} dr \right)(\xi) +\epsilon  \, .
\end{align*}
Define an auxiliary star body $D$ by
$$ \|x\|_D^{-\kappa} = \|x\|_K^{-\kappa} \frac{f(x\|x\|_K^{-1})}{g(x\|x\|_K^{-1})} \, .$$
Note that $D$ is an even $\kappa$-balanced star body and $\|\cdot\|_D^{-\kappa}$ is positive-definite, thus $D$ is an intersection body in $\K^n$. By Definition \ref{df_ib}, there is a measure $\nu$ on $\s^{\kappa n -1}$ corresponding to the body $D$. Integrating the above inequality over the sphere with respect to the measure $\nu$ and applying the equality condition of Definition \ref{df_ib}, yields
\begin{align} 
	& \int\limits_{\s^{\kappa n -1}} \|x\|_D^{-\kappa} \int\limits_0^{\|x\|_K^{-1}} g(rx) r^{\kappa n-\kappa -1} dr  dx \label{eq_bpameqone} \\ 
	&\leq \int\limits_{\s^{\kappa n -1}} \|x\|_D^{-\kappa} \int\limits_0^{\|x\|_L^{-1}} g(rx) r^{\kappa n-\kappa -1} dr  dx +\epsilon \int\limits_{\s^{\kappa n -1}} d\nu(\xi) \nonumber \, .
\end{align}
By Lemma \ref{lemma_elem}, with $a=\|x\|_K^{-1}, b=\|x\|_L^{-1}, \alpha(r)=f(rx), \beta(r)=g(rx)$, we also have
\begin{align} 
	&\int\limits_0^{\|x\|_K^{-1}} f(rx) r^{\kappa n-1} dr - \|x\|_D^{-\kappa} \int\limits_0^{\|x\|_K^{-1}} g(rx) r^{\kappa n-\kappa -1} dr  \label{eq_bpameqtwo} \\ 
	&\leq \int\limits_0^{\|x\|_L^{-1}} f(rx) r^{\kappa n-1} dr - \|x\|_D^{-\kappa} \int\limits_0^{\|x\|_L^{-1}} g(rx) r^{\kappa n-\kappa -1} dr  \nonumber \, .
\end{align}
Integrating equation (\ref{eq_bpameqtwo}) over the sphere and adding the resulting equation to equation (\ref{eq_bpameqone}), we obtain
\begin{align*} 
	 \int\limits_{\s^{\kappa n -1}} \int\limits_0^{\|x\|_K^{-1}} f(rx) r^{\kappa n-1} dr dx \leq \int\limits_{\s^{\kappa n -1}} \int\limits_0^{\|x\|_L^{-1}} f(rx) r^{\kappa n-1} dr dx +\epsilon \int\limits_{\s^{\kappa n -1}} d\nu(\xi) \, ,
\end{align*}
which reads as
$$ \mu(K) \leq \mu(L) + \epsilon \int\limits_{\s^{\kappa n -1}} d\nu(\xi) \, .$$
Finally, since the spherical Radon transform on $\K^n$ of the constant function one, is the constant function with value $|\s^{\kappa n -\kappa-1}|$, using the equality condition of Definition \ref{df_ib}, we obtain
\begin{align*}
	\int\limits_{\s^{\kappa n -1}} d\nu(\xi)
	&= \frac{1}{|\s^{\kappa n -\kappa-1}|} \int\limits_{\s^{\kappa n -1}} \RT^{\kappa} 1(\xi) d\nu(\xi) \\
	&= \frac{1}{|\s^{\kappa n -\kappa-1}|} \int\limits_{\s^{\kappa n -1}} \|x\|_D^{-\kappa} dx \\
	&= \frac{1}{|\s^{\kappa n -\kappa-1}|} \int\limits_{\s^{\kappa n -1}} \|x\|_K^{-\kappa} \frac{f(x\|x\|_K^{-1})}{g(x\|x\|_K^{-1})} dx \, .
\end{align*}
\end{proof}

\begin{pr}\label{pr_negative}
Let $f,g$ be even strictly positive $\kappa$-invariant continuous functions on $\R^{\kappa n}\setminus \{0\}$ and so that for any fixed $x\in\s^{\kappa n -1}$, $f(tx), g(tx)$ are locally-integrable in $t$ and $t^{\kappa} \frac{f(tx)}{g(tx)}$ is a non-decreasing function in $t$. Let $l=\max\{2, \kappa-2\}$ and assume also that $g\in C^l(\R^{\kappa n} \setminus \{0\})$. Suppose $L$ is an infinitely-smooth origin-symmetric convex body in $\K^n$ with strictly positive curvature so that 
\begin{equation}\label{eq_bpamnpd}
\|x\|_L^{-\kappa} \frac{f(x\|x\|_L^{-1})}{g(x\|x\|_L^{-1})}
\end{equation}
is in $C^{\kappa n -\kappa}(\R^{\kappa n} \setminus \{0\})$ and does not represent a positive-definite distribution on $\R^{\kappa n}$. Then there is an origin-symmetric convex body $K$ in $\K^n$ satisfying
$$ \gamma(K\cap H_{\xi}) \leq \gamma(L\cap H_{\xi}) \, ,$$
for every $\xi\in\s^{\kappa n -1}$, but
$$ \mu(K) > \mu(L) \, .$$ 
\end{pr}

\begin{proof}
Since the function (\ref{eq_bpamnpd}) is in $C^{\kappa n -\kappa -1}(\R^{\kappa n} \setminus \{0\})$, it follows by Corollary 3.17 (i) in \cite{Koldobsky2005}, that its Fourier transform is a continuous function on the sphere. Hence, by continuity, its Fourier transform must be negative on some open subset $\Omega$ of the sphere. From the $\kappa$-invariance of the function (\ref{eq_bpamnpd}), it follows that the set $\Omega$ is $\kappa$-balanced. Let $h$ be an infinitely-smooth non-negative and not identically zero $\kappa$-invariant function on the sphere with support contained in the set $\Omega$. Extend $h$ to a homogeneous function of degree $-\kappa$, then the Fourier transform of this extension is a homogeneous function of degree $-\kappa n +\kappa$, i.e. there is an infinitely smooth function $v$ on the sphere so that $(h \cdot r^{-\kappa})^{\wedge}=v \cdot r^{-\kappa n+\kappa}$.

\noindent
Let $\epsilon >0$, define another body $K$ by
\begin{align*}
	&|x|_2^{-\kappa n+\kappa} \int_0^{\frac{|x|_2}{\|x\|_K}} r^{-\kappa n+\kappa-1} g\left(\frac{rx}{|x|_2}\right) dr \\
	&=|x|_2^{-\kappa n+\kappa}\int_0^{\frac{|x|_2}{\|x\|_L}} r^{-\kappa n+\kappa-1} g\left(\frac{rx}{|x|_2}\right) dr 
-\epsilon |x|_2^{-\kappa n+\kappa} v\left(\frac{x}{|x|_2}\right) \, .
\end{align*}
As $g \in C^2(\R^{\kappa n} \setminus \{0\})$, by Lemma 5.16 in \cite{Koldobsky2005}, $K$ is convex for $\epsilon$ small enough. Since the function $h$ is positive, using equation (\ref{eq_msresection}), it follows  
\begin{align*}
	&\gamma(K\cap H_{\xi}) \\
	&=\frac{|S^{\kappa-1}|}{(2 \pi)^{\kappa}} \left(|x|_2^{- \kappa n + \kappa} \int_0^{\frac{|x|_2}{\|x\|_K}} g\left(\frac{rx}{|x|_2}\right) r^{\kappa n-\kappa -1} dr \right)^{\wedge}(\xi) \\
	&=\frac{|S^{\kappa-1}|}{(2 \pi)^{\kappa}} \left(|x|_2^{- \kappa n + \kappa} \int_0^{\frac{|x|_2}{\|x\|_L}} g\left(\frac{rx}{|x|_2}\right) r^{\kappa n-\kappa -1} dr \right)^{\wedge}(\xi)
-\epsilon \frac{|S^{\kappa-1}|(2 \pi)^{\kappa n}}{(2 \pi)^{\kappa}} h(\xi)\\
	&\leq \gamma(L\cap H_{\xi})  \, .
\end{align*}

On the other hand, the function $h$ is supported on the set where the Fourier transform of the function (\ref{eq_bpamnpd}) is negative, hence
\begin{align*}
	& \left(\|x\|_L^{-\kappa} \frac{f(x\|x\|_L^{-1})}{g(x\|x\|_L^{-1})}\right)^{\wedge}(\xi) \left(|x|_2^{- \kappa n + \kappa} \int_0^{\frac{|x|_2}{\|x\|_K}} g\left(\frac{rx}{|x|_2}\right) r^{\kappa n-\kappa -1} dr \right)^{\wedge}(\xi) \\
	&= \left(\|x\|_L^{-\kappa} \frac{f(x\|x\|_L^{-1})}{g(x\|x\|_L^{-1})}\right)^{\wedge}(\xi) \left(|x|_2^{- \kappa n + \kappa} \int_0^{\frac{|x|_2}{\|x\|_L}} g\left(\frac{rx}{|x|_2}\right) r^{\kappa n-\kappa -1} dr \right)^{\wedge}(\xi) \\
& \,\,\,\,\,\,\,\,\,\,\,\,\,\,\,\,\,\,\,\,\,\,\,\,\,\,\,\,\,\,\,\, -\epsilon (2 \pi)^{\kappa n} \left(\|x\|_K^{-\kappa} \frac{f(x\|x\|_K^{-1})}{g(x\|x\|_K^{-1})}\right)^{\wedge}(\xi) h(\xi) \\
	&\geq \left(\|x\|_L^{-\kappa} \frac{f(x\|x\|_L^{-1})}{g(x\|x\|_L^{-1})}\right)^{\wedge}(\xi) \left(|x|_2^{- \kappa n + \kappa} \int_0^{\frac{|x|_2}{\|x\|_L}} g\left(\frac{rx}{|x|_2}\right) r^{\kappa n-\kappa -1} dr \right)^{\wedge}(\xi) \, .
\end{align*}
Note that the above inequality is strict on $\Omega$. 

Since $g\in C^{\kappa-2}(\R^{\kappa n} \setminus \{0\})$, by Corollary 3.17 (i) in \cite{Koldobsky2005}, functions $\xi\mapsto \gamma(K\cap H_{\xi})$ and $\xi\mapsto \gamma(L\cap H_{\xi})$ are continuous positive functions on the sphere. Integrating the latter inequality over the sphere and applying the spherical Parseval's formula in the form of Corollary 3.23 in \cite{Koldobsky2005} with $k=\kappa n-\kappa$, which is justified by above observations and the fact that the function (\ref{eq_bpamnpd}) is in $C^{\kappa n -\kappa}(\R^{\kappa n} \setminus \{0\})$, we obtain
\begin{align*}
	& \int_{\s^{\kappa n-1}} \|x\|_L^{-\kappa} \frac{f(x\|x\|_L^{-1})}{g(x\|x\|_L^{-1})} \int_0^{\|x\|_K^{-1}} g(rx) r^{\kappa n-\kappa -1} dr dx \\
	&> \int_{\s^{\kappa n-1}} \|x\|_L^{-\kappa} \frac{f(x\|x\|_L^{-1})}{g(x\|x\|_L^{-1})} \int_0^{\|x\|_L^{-1}} g(rx) r^{\kappa n-\kappa -1} dr dx \, .
\end{align*}
This is equivalent to 
\begin{equation}\label{eq_nepeq1}
 0< \int_{\s^{\kappa n-1}} \|x\|_L^{-\kappa} \frac{f(x\|x\|_L^{-1})}{g(x\|x\|_L^{-1})} \int_{\|x\|_L^{-1}}^{\|x\|_K^{-1}} g(rx) r^{\kappa n-\kappa -1} dr dx \, .
\end{equation}

By Lemma \ref{lemma_elem}, with $a=\|x\|_L^{-1}, b=\|x\|_K^{-1}, \alpha(r)=f(rx), \beta(r)=g(rx)$, we also have
\begin{equation}\label{eq_nepeq2}
	\|x\|_L^{-\kappa} \frac{f(x\|x\|_L^{-1})}{g(x\|x\|_L^{-1})} \int_{\|x\|_L^{-1}}^{\|x\|_K^{-1}} g(rx) r^{\kappa n-\kappa -1} dr  
	\leq \int_{\|x\|_L^{-1}}^{\|x\|_K^{-1}} f(rx) r^{\kappa n-1} dr \, .
\end{equation}

Integrating equation (\ref{eq_nepeq2}) over the sphere and combining the resulting equation with inequality (\ref{eq_nepeq1}), yields
$$ \int_{\s^{\kappa n-1}} \int_0^{\|x\|_L^{-1}} f(rx) r^{\kappa n-1} dr dx < \int_{\s^{\kappa n-1}} \int_0^{\|x\|_K^{-1}} f(rx) r^{\kappa n-1} dr dx \, ,$$
which is equivalent to
$$ \mu(L) < \mu(K) \, . $$
\end{proof}

\begin{theorem}\label{th_bppokn}
Let $f=g$ be equal even non-negative $\kappa$-invariant continuous functions on $\R^{\kappa n}\setminus \{0\}$ that are locally-integrable on every line through the origin. Then the answer to the Busemann-Petty problem in $\K^n$ for arbitrary measures is positive in the following cases: (i) $n=2, \kappa\in \N$, (ii) $n=3, \kappa\leq 2$ and (iii) $n=4, \kappa =1$. In the remaining cases the answer to the Busemann-Petty problem in $\K^n$ for arbitrary measures is negative for an even strictly positive  $\kappa$-invariant continuous function $f\in C^l(\R^{\kappa n} \setminus \{0\})$ with $l=\max\{2, \kappa-2\}$. 
\end{theorem}

\begin{proof}
Since $t^{\kappa} \frac{f(tx)}{g(tx)} = t^{\kappa}$ is a non-decreasing function, Propositions \ref{pr_positive} and \ref{pr_negative} apply. Suppose $K$ is an intersection body in $\K^n$, then $\|x\|_K^{-\kappa} \frac{f(x\|x\|_K^{-1})}{g(x\|x\|_K^{-1})} = \|x\|_K^{-\kappa}$ is a positive-definite distribution on $\R^{\kappa n}$. The affirmative part now follows from Corollary \ref{co_ib} and Proposition \ref{pr_positive} with $\epsilon = 0$.

For the negative part, note that in this case there is an origin-symmetric convex body $L$ in $\K^n$ that is not an intersection body in $\K^n$, e.g. $B^{\kappa n}_q$ with $q>2$, see Section \ref{section_cibinkn}. $L$ can be approximated in the radial metric by a sequence of infinitely-smooth origin-symmetric convex bodies $L_m$ in $\K^n$ with strictly positive curvature so that each body $L_m$ is not an intersection body in $\K^n$. This follows from Lemma 4.10 in \cite{Koldobsky2005} and the connection between the convolution and linear transformations. Thus we can assume that $\|x\|_L^{-\kappa} \frac{f(x\|x\|_L^{-1})}{g(x\|x\|_L^{-1})} = \|x\|_L^{-\kappa}$ is in $C^{\infty}(\R^{\kappa n} \setminus \{0\})$ and does not represent a positive definite distribution. The negative part now follows from Proposition \ref{pr_negative}.
\end{proof}

The volume difference inequality is obtained by interchanging the roles of $K$ and $L$ in Proposition \ref{pr_positive}.
\begin{co}\label{co_vdiam}
Under the assumptions of Proposition \ref{pr_positive}, we have 
\begin{align*} 
|\mu(K)-& \mu(L)| \leq \frac{1}{|\s^{\kappa n -\kappa-1}|} \max\limits_{\xi\in\s^{\kappa n-1}} |\gamma(K\cap H_\xi)-\gamma(L\cap H_\xi)| \times \\
& \times \max\left\{ \int_{\s^{\kappa n-1}} \|x\|_{K}^{-\kappa} \frac{f(x\|x\|_{K}^{-1})}{g(x\|x\|_{K}^{-1})} dx, \int_{\s^{\kappa n-1}} \|x\|_{L}^{-\kappa} \frac{f(x\|x\|_{L}^{-1})}{g(x\|x\|_{L}^{-1})} dx \right\} \, .
\end{align*}
\end{co}

\begin{theorem}\label{th_hifam}
Let $f=g$ be equal even non-negative $\kappa$-invariant continuous functions on $\R^{\kappa n}\setminus \{0\}$ that are locally-integrable on every line through the origin. Let $K$ be an intersection body in $\K^n$, then
\begin{equation*} 
\mu(K) \leq \frac{n}{n-1} \frac{|B^{\kappa n}_2|^{\frac{n-1}{n}}}{|B_2^{\kappa n - \kappa}|} \max\limits_{\xi\in\s^{\kappa n-1}} \gamma(K\cap H_\xi) |K|^{\frac{1}{n}} \, .
\end{equation*}
\end{theorem}

\begin{proof}
Let $f=g$ in the inequality of Corollary \ref{co_vdiam}. Further, set $L=\delta B^{\kappa n}_2$, let $\delta$ go to zero, and observe that by H{\"o}lder's inequality
\begin{align*}
 \frac{1}{|\s^{\kappa n -\kappa-1}|} \int_{\s^{\kappa n-1}} \|x\|_{K}^{-\kappa} dx  
 &\leq \frac{1}{|\s^{\kappa n -\kappa-1}|} \left( \int_{\s^{\kappa n-1}} \|x\|_{K}^{-\kappa n} dx \right)^{\frac{1}{n}} |S^{\kappa n -1}|^{\frac{n-1}{n}} \\
 &=\frac{n}{n-1} \frac{|B^{\kappa n}_2|^{\frac{n-1}{n}}}{|B_2^{\kappa n - \kappa}|} |K|^{\frac{1}{n}}.
\end{align*}
The constant is the best possible, this follows by a similar example as in \cite{MR2891246}, Theorem 1.
\end{proof}

For $\kappa=1,2$ the inequality of Theorem \ref{th_hifam} reduces to the previously known hyperplane inequalities for arbitrary measures, see \cite{MR2891246} and \cite{MR3106735}.

\begin{lemma}\label{le_hiamdfg}
Let $M$ be an intersection body in $\K^n$ and let $K$ be any star body in $\K^n$, then
\begin{equation*} 
	\int_K \|x\|_M^{-l-\kappa} dx \leq \frac{n}{n-1} \frac{|B^{\kappa n}_2|^{\frac{n-1}{n}}}{|B_2^{\kappa n - \kappa}|} |M|^{\frac{1}{n}} \max\limits_{\xi\in\s^{\kappa n-1}} \int_{K\cap H_{\xi}} \|x\|_M^{-l} dx  \, ,
\end{equation*}
for $l<\kappa n -\kappa$.
\end{lemma}

\begin{proof}
Let $f(x)=\|x\|_M^{-l-\kappa}$ and $g(x)=\|x\|_M^{-l}$, then $t^{\kappa} \frac{f(tx)}{g(tx)} = \|x\|_M^{-\kappa}$ is a non-decreasing function, $\|x\|_K^{-\kappa} \frac{f(x\|x\|_K^{-1})}{g(x\|x\|_K^{-1})} = \|x\|_M^{-\kappa}$ is a positive-definite distribution and hence Corollary \ref{co_vdiam} applies. The result follows by setting $L=\delta B^{\kappa n}_2$ and letting $\delta$ go to zero. 
\end{proof}

\noindent
Setting $l=-\kappa$ and $M=B^{\kappa n}_2$ in Lemma \ref{le_hiamdfg}, yields 
\begin{co}
For any star body $K$ in $\K^n$, we have
\begin{equation*} 
	|K| \leq \frac{n}{n-1} \frac{|B^{\kappa n}_2|}{|B_2^{\kappa n - \kappa}|} \max\limits_{\xi\in\s^{\kappa n-1}} \int_{K\cap H_{\xi}} |x|_2^{\kappa} dx  \, .
\end{equation*}
\end{co}

\noindent
And setting $l=0$ and $M=B^{\kappa n}_2$ in Lemma \ref{le_hiamdfg}, we obtain
\begin{co}
For any star body $K$ in $\K^n$, we have
\begin{equation*} 
	\int_K |x|_2^{-\kappa} dx \leq \frac{n}{n-1} \frac{|B^{\kappa n}_2|}{|B_2^{\kappa n - \kappa}|} \max\limits_{\xi\in\s^{\kappa n-1}} |K\cap H_{\xi}| \, .
\end{equation*}
\end{co}

\section{Intersection bodies of convex bodies in $\K^n$}\label{section_IntersectionBodiesOfConvexBodies}

In this section we extend to $\K^n$ Busemann's theorem, which says that the intersection body of an origin-symmetric convex body in $\R^n$ is convex. 

The first part of the proof goes along the lines of the proof of Busemann's theorem in $\R^n$, up to inequality (\ref{eq_psatte}). Busemann's theorem in $\R^n$ is then obtained by applying the arithmetic-geometric mean inequality. This step has to be replaced by the use of a result of K. Ball, as it was done in the complex case, see \cite{MR3106735}. We will use the following result of Ball, as stated in \cite{MR3106735}.

\begin{pr}\textnormal{(\cite{MR3106735} , Corollary 7.5)}\label{pr_ballpaouris}
Let $r_1, r_2 >0$ and let $\alpha >0$. Define $\lambda, r_3$ as follows:
$$  \lambda=\frac{r_1}{r_1+r_2}, \,\,\,\,\, r_3=\frac{\alpha}{r_1^{-1}+r_2^{-1}}  \, .$$
Assume that $f_1, f_2, f_3 : [0, \infty) \rightarrow [0, \infty)$ such that
$f_3(r_3) \geq f_1(r_1)^{(1-\lambda)} f_2(r_2)^{\lambda}$ for any $r_1, r_2 >0$.
Let $p\geq 1$ and denote 
$$ A^p=\int_0^{\infty} f_1(r)r^{p-1}dr, \,\,\,\,\, B^p=\int_0^{\infty} f_2(r)r^{p-1}dr, \,\,\,\,\, C^p=\int_0^{\infty} f_3(r)r^{p-1}dr \,.$$
Then 
$$ C \geq \frac{\alpha}{\frac{1}{A}+\frac{1}{B}} \, .$$
\end{pr}

\begin{theorem}\textnormal{(Busemann's theorem in $\K^n$)}\label{th_bt}
Let $S$ be a $\kappa(n-2)$-di\-men\-sional $\kappa$-balanced subspace of $\R^{\kappa n}$ and $u\in \s^{\kappa n-1} \cap S^{\perp}$. Denote by $S_u= span \left\{ S, H_{u}^{\perp} \right\}$. Define a function $r: \s^{\kappa n-1} \cap S^{\perp} \rightarrow (0, \infty)$ by 
$$ r(u)=|K\cap S_u|^{1/\kappa} \, .$$
Then the curve $r$ is the boundary of a $\kappa$-balanced convex body in $S^{\perp}$.
\end{theorem}

\begin{proof}
The curve $r$ is the boundary of a convex body in $S^{\perp}$ if and only if $r^{-1}$ satisfies the triangle inequality. Thus it is enough to show that for two linearly independent unit vectors $u_1, u_2$ in $S^{\perp}$ and $u_3=\frac{u_1+u_2}{|u_1+u_2|}$, we have
\begin{equation}\label{eq_tifr}
\frac{|u_1+u_2|}{r(u_3)} \leq \frac{1}{r(u_1)}+\frac{1}{r(u_2)} \, .
\end{equation}
We may assume that $H_{u_1}^{\perp} \cap H_{u_2}^{\perp} = \{0\}$, otherwise $H_{u_1}^{\perp} = H_{u_2}^{\perp}$ and (\ref{eq_tifr}) is trivially satisfied since $K\cap H_{u_1}^{\perp} \subset S^{\perp}$ is a ball. 

Let $r_j>0$, $j=1,2$, and let $r_3 u_3=(1-\lambda) r_1 u_1 + \lambda r_2 u_2$ be the intersection point of the line in the direction $u_3$ with the line segment with endpoints $r_1 u_1, r_2 u_2$, then 
$$ \lambda=\frac{r_1}{r_1+r_2}, \,\,\,\,\, \frac{r_3}{|u_1+u_2|}=\frac{1}{r_1^{-1}+r_2^{-1}} \, .$$ 
For $t>0$, let $f_{u_j}(t)= |K\cap (S+t u_j)|$, $1\leq j \leq 3$. Observe that $f_{u_j}(t) = f_{R_{\sigma}(u_j)}(t)$ for any $\sigma \in SO(\kappa)$. Indeed, since $K$ is $\kappa$-balanced
$$ f_{u_j}(t)	= \int\limits_{S+t u_j} \chi(\|x\|_K) dx = \int\limits_{S+t R_{\sigma}(u_j)} \chi(\| x\|_K) dx = f_{R_{\sigma}(u_j)}(t) \, .$$
This, in turn, implies that 
$$ r(u_j) = \left( |S^{\kappa-1}| \int_0^{\infty} f_{u_j}(t) t^{\kappa-1} dt \right)^{1/\kappa}, \,\,\, 1\leq j \leq 3 \, ,$$
since 
\begin{align*}
r^{\kappa}(u_j) 	&= \int_{H_{u_j}^{\perp}} |K\cap (S+x)| dx\\
				&= \int_0^{\infty} \int_{\s^{\kappa n-1}\cap H_{u_j}^{\perp}} |K\cap (S+t\theta)| d\theta \, t^{\kappa-1} dt \\
				&= \int_0^{\infty} \int_{\s^{\kappa n-1}\cap H_{u_j}^{\perp}} f_{\theta}(t) d\theta \, t^{\kappa-1} dt \\
				&= |S^{\kappa-1}| \int_0^{\infty} f_{u_j}(t) d\theta\, t^{\kappa-1} dt \, .
\end{align*}
Note that $r$ is $\kappa$-invariant.

By construction the sets $K\cap (S+r_j u_j)$, $1\leq j\leq 3$ lie in an affine subspace of $\R^{\kappa n}$. Hence, by convexity of $K$, for $\lambda$ as defined above 
$$ (1-\lambda) (K\cap (S+r_1 u_1)) + \lambda (K\cap (S+r_2 u_2)) \subset K\cap (S+r_3 u_3) \, .$$ 
Applying the Brunn-Minkowski inequality, we obtain
$$ f_{u_3}(r_3)^{1/\kappa(n-2)} \geq (1-\lambda) f_{u_1}(r_1)^{1/\kappa(n-2)} + \lambda f_{u_2}(r_2)^{1/\kappa(n-2)} \, ,$$
and the arithmetic-geometric mean inequality yields
\begin{equation}\label{eq_psatte}
f_{u_3}(r_3) \geq f_{u_1}(r_1)^{(1-\lambda)} f_{u_2}(r_2)^{\lambda} \, .
\end{equation}
Now we apply Proposition \ref{pr_ballpaouris} with $p=\kappa$ and $\alpha=|u_1+u_2|$, this gives what we need
$$ \frac{|u_1+u_2|}{r(u_3)} \leq \frac{1}{r(u_1)}+\frac{1}{r(u_2)} \, . $$
\end{proof}

\begin{co}\label{co_ibocb}
Let $K$ be an origin-symmetric convex body in $\K^n$, then $I_{\K}(K)$ is also an origin-symmetric convex body in $\K^n$.
\end{co}

\begin{proof}
In case $n=2$, $H_{\xi}$ is $\kappa$-dimensional and hence $K\cap H_{\xi}$ is a ball. This implies that $I_{\K}(K)$ is a rotation of $K$. Indeed, let $\xi\in\s^{\kappa n-1}$, then
$$ \frac{|S^{\kappa-1}|}{\kappa} \|\xi\|_{I_{\K}(K)}^{-\kappa} = |K\cap H_{\xi}| = \frac{|S^{\kappa-1}|}{\kappa} \|x\|_K^{-\kappa} \, , $$
for any $x\in K\cap H_{\xi}$.

Now assume $n\geq 3$. A subset $L$ of $\R^{\kappa n}$ is convex if and only if all its two-dimensional sections through any fixed point are convex. In particular, if for any linearly independent vectors $x,y$, the section $L\cap \mathrm{span}$ $\!\!\{x,y\}$ is convex. The condition that $L\cap \mathrm{span}$ $\!\!\{H_x^{\perp}, H_y^{\perp}\} $ is convex, is stronger and hence implies that L is convex.

Let $S$ be a $\kappa(n-2)$-dimensional $\kappa$-balanced subspace of $\R^{\kappa n}$ and $u, v \in \s^{\kappa n-1} \cap S^{\perp}$ so that $v \perp H_u^{\perp}$. Observe that by definition of intersection bodies in $\K^n$,
$$ |K\cap H_v| = |I_{\K}(K) \cap H_v^{\perp}| = \frac{|S^{\kappa-1}|}{\kappa} \|v\|_{I_{\K}(K)}^{-\kappa} = \frac{|S^{\kappa-1}|}{\kappa} \rho_{I_{\K}(K)}^{\kappa}(v)\, .$$
Hence, in the notation of Busemann's theorem
$$ \rho^{\kappa}_{I_{\K}(K)}(v) = \frac{\kappa}{|S^{\kappa-1}|} |K\cap H_v| = \frac{\kappa}{|S^{\kappa-1}|} |K\cap S_u| = \frac{\kappa}{|S^{\kappa-1}|} r(u) \, .$$
This shows that $I_{\K}(K) \cap S^{\perp}$ is convex, and hence $I_{\K}(K)$ is convex.
\end{proof}

\begin{re}
Together Corollaries \ref{co_ibkib} and \ref{co_ibocb} show that $\kappa$-intersection bodies of $\kappa$-balanced convex bodies in $\R^{\kappa n}$ exist and are convex. This is not true in general. In \cite{yaskin2013} V. Yaskin constructed an origin-symmetric convex body in $\R^n$ such that its $k$-intersection body exists and is not convex, for $2 \leq k\leq n-1$.
\end{re}

The result of D. Hensley \cite{hensley80} and C. Borell \cite{borell73}, that the intersection body of a convex body is isomorphic to an ellipsoid, extends to $\K^n$ via a result from \cite{MR2826412}. Recall that the Banach-Mazur distance of two origin-symmetric convex bodies $K,L$ in $\R^n$ is defined as
$$ d_{BM}(K,L) = \inf\{a>0 \: K\subset T L \subset a K\ \text{ with } T \in GL_n\}\,.$$

\begin{pr}\textnormal{(\cite{MR2826412}, Theorem 1.2)}\label{pr_ipoib}
Let $K$ be an origin-symmetric convex body in $\R^n$ and assume that the $k$-intersection body of $K$, $I_k(K)$, exists and is convex, then 
$$ d_{BM}(I_k(K),B^n_2) \leq c(k) \, ,$$
where $c(k)$ only depends on $k$.
\end{pr}

Combining the above proposition with Corollaries \ref{co_ibocb} and \ref{co_ibkib} yields
\begin{co}\label{co_hbt}
Let $K$ be an origin-symmetric convex body in $\K^n$, then 
$$ d_{BM}(I_{\K}(K),B^{\kappa n}_2) \leq c(\kappa) \, ,$$
where $c(\kappa)$ only depends on $\kappa$.
\end{co}

\section*{Acknowledgements}
The first named author thanks the Oberwolfach Research Institute for Mathematics for its hospitality and support, where part of this work was carried out. The second named author would like to acknowledge support from the programme ``API$\Sigma$TEIA II-ΑΤΟCB-3566" of the General Secretariat for Research
and Technology of Greece.

\nocite{}
\bibliographystyle{amsplain}
\bibliography{ref_01wc}

\end{document}